\documentclass[11pt,reqno]{amsart}

\usepackage{amssymb}
\usepackage{color}
\usepackage{fullpage}
\usepackage{hyperref}
\usepackage{pgf,tikz}
\usepackage{float}
\usepackage{caption}
\usetikzlibrary{arrows}

\newtheorem{theorem}{Theorem}[section]

\newtheorem{corollary}[theorem]{Corollary}
\newtheorem{conjecture}[theorem]{Conjecture}
\newtheorem{proposition}[theorem]{Proposition}

\newtheorem{remark}[theorem]{Remark}

\newtheorem{notation}[theorem]{Notation}

\DeclareMathOperator{\Tor}{Tor}
\DeclareMathOperator{\pdim}{pdim}
\DeclareMathOperator{\depth}{depth}

\newcommand{\m}{\mathfrak{m}}

\newcommand{\reg}{\operatorname{reg}}
\renewcommand{\span}{\operatorname{Span}}

\begin{document}
\title{Syzygies, Betti numbers and regularity of cover ideals of certain multipartite graphs}
\author{A. V. Jayanthan}
\email{jayanav@iitm.ac.in}
\address{Department of Mathematics, Indian Institute of Technology
Madras, Chennai, INDIA - 60036}
\author{Neeraj Kumar}
\email{neeraj@iith.ac.in, neeraj.unix@gmail.com}
\address{Department of Mathematics, Indian Institute of Technology
Hyderabad, Kandi, Sangareddy, INDIA - 502285}
\begin{abstract}
Let $G$ be a finite simple graph on $n$ vertices. Let $J_G \subset
K[x_1, \ldots, x_n]$ be the cover ideal of $G$. In this article, we
obtain syzygies, Betti numbers and Castelnuovo-Mumford regularity of
$J_G^s$ for all $s \geq 1$ for certain classes of graphs $G$.
\end{abstract}
\maketitle

\section{Introduction}
Computation of minimal free resolution and syzygies of ideals and
modules over polynomial rings have always attracted researchers in
commutative algebra and algebraic geometry. Recently, there have been
a lot of interest in studying the homological aspects of squarefree
monomial ideals in polynomial rings. Since these ideals have strong
combinatorial connections, problems in this area have attracted both
commutative algebrists and combinatorists. Even in this case, there
are only very few cases of ideals for which explicit computation of
the resolution, including the complete description of syzygies, is
done. The main open problems in this area are to find/construct
minimal free resolutions in more cases, and to introduce new ideas and
structures, \cite[Remark 2.6]{ps08}. There are even less results on
the resolution of powers of ideals. If $I$ is generated by a regular
sequence, then $I^s$ is a determinantal ideal and hence the minimal
free resolution of $I^s$ can be obtained by using the Eagon-Northcott
complex, \cite{be75}. Since the maps in this resolution may not be degree
preserving, the computation of graded Betti numbers and more generally
the computation of the syzygies may not be possible. In
\cite{guardo-vtuyl05}, they explicitly compute the graded Betti numbers
of $I^s$ if $I$ is a homogeneous complete intersection in a polynomial
ring.

Among the resolutions, linear resolutions are possibly the simplest to
describe. Let $R = K[x_1,\ldots,x_n]$, where $K$ is a field. If $I$ is
the defining ideal of the rational normal curve in $\mathbb{P}^{n-1}$,
then Conca proved that $I^s$ has a linear resolution for all $s$,
\cite{conca00}. Herzog, Hibi and Zheng proved that if $I$ is a
monomial ideal generated in degree $2$, then $I$ has linear resolution
if and only if $I$ has linear quotients if and only if $I^s$ has a
linear resolution for all $s \geq 1$, \cite{hhz04}. It was proved by
Conca and Herzog, \cite{ch03}, that if $I$ is a polymatroidal ideal in
$R$, then $I$ has linear quotients. Moreover, they proved that the
product of polymatroidal ideals are polymatroidal. Therefore, if $I$
is polymatroidal, then $I^s$ has linear resolution for all $s \geq 1$.
If $I$ is a lexsegment ideal, then Ene and Olteanu proved that $I$ has
linear resolution if and only if $I$ has linear quotients if and only
if $I^s$ has linear resolution for all $s \geq 1$ if and only if $I^s$
has linear quotients for all $s \geq 1$, \cite{ene-olteanu12}.

In all the above mentioned results, the authors do not compute the
syzygy modules. In general, it is a non-trivial task to compute the
syzygy modules, even when the resolution is linear. In this article,
we compute the resolution, syzygies and Betti numbers of powers of
certain classes of squarefree monomial ideals. As a by-product, we
obtain expression for the regularity of powers of these classes of
ideals. It was shown by
Kodiyalam \cite{vijay} and independently by Cutkosky, Herzog and Trung
\cite{CHT99} that if $I$ is a homogeneous ideal in a polynomial ring,
then there exist non-negative integers $d, e$ and $s_0$ such that
$\reg(I^s) = ds + e$ for all $s \geq s_0$, where $\reg(-)$ denote the
Castelnuovo-Mumford regularity. Kodiyalam proved that $d \leq
\deg(I)$, where $\deg(I)$ denotes the largest degree of a homogeneous
minimal generator of $I$. In general, the stability index $s_0$ and
the constant term $e$ are hard to compute. There have been discrete
attempts in identifying $s_0$ and $e$ for certain classes of ideals.

Recently, there have been a lot of activity in studying the interplay
between the combinatorial properties of graphs and the algebraic
properties of ideals associated to graphs.  For a finite simple graph
$G$ on the vertex set $\{x_1, \ldots, x_n\}$, let $J_G \subset R =
K[x_1, \ldots, x_n]$, where $K$ is a field, denote the cover ideal of
$G$ (see Section 2 for the definition). It may be noted that $J_G =
\cap_{\{x_i,x_j\} \in E(G)} (x_i, x_j)$ is the Alexander dual of the
edge ideal of $I$ (see Section 2 for the definition). While the
connection between algebraic properties of the edge ideal and
combinatorial properties of the graph has been studied extensively,
not much is known about the connection between the properties of the
cover ideal and the graph.
\iffalse
Given a finite simple graph, one can identify the vertices with
indeterminates and associate ideals in a polynomial ring, for example,
edge ideals and cover ideals (see Section 2 for definition)
corresponding to the given graph.
For edge ideals of various classes of
graphs, $e$ and $s_0$ have been computed, see for example
\cite{banerjee, huneke, froberg, ha_adam, sean_thesis, jns17,
  khosh_moradi, kummini, mohammad, nevo_peeva,adam, Villarreal, russ,
Zheng}. In the case of edge ideals, $d = 2$ and in the known cases,
$e$ is connected to combinatorial invariants associated with the graph
$G$. 
\fi
In \cite{fakhari16},
Seyed Fakhari studied certain homological properties of symbolic
powers of cover ideals of very well-covered and bipartite graphs. It
was shown that if $G$ is a very well-covered graph and $J_G$ has a
linear resolution, then $J_G^{(s)}$ has a linear resolution for all $s
\geq 1$. Furthermore, it was proved that if $G$ is a bipartite graph
with $n$ vertices, then for $s \geq 1$, \[ \reg(J_G^s) \leq s\deg(J_G)
+ \reg(J_G) + 1.\] Hang and Trung, in \cite{ht17}, studied
unimodular hypergraphs and proved that if $\mathcal{H}$ is a
unimodular hypergraph on $n$ vertices and rank $r$ and
$J_{\mathcal{H}}$ is the cover ideal of $\mathcal{H}$, then there
exists a non-negative integer $e \leq \dim(R/J_{\mathcal{H}}) -
\deg(J_{\mathcal{H}}) + 1$ such that \[reg J_{\mathcal{H}}^s =
\deg(J_{\mathcal{H}}) s + e\] for all $s \geq \frac{rn}{2} + 1$.
Since bipartite graphs are unimodular, their results hold true in
the case of bipartite graphs as well. While the first result gives
an upper bound for the constant term, the later result gives the
upper bound for both the stability index and the constant term. 

In this article, we
obtain the complete description of the minimal free resolution,
including the syzygies, of $J_G^s$ for some classes of multipartite
graphs. The
paper is organized as follows. In Section 2, we collect the
preliminaries required for the rest of the paper. We study the
resolution of powers of cover ideals of certain bipartite graphs in
Section 3.  If $G$ is a complete bipartite graph, then $J_G$ is a
regular sequence and hence the minimal graded free resolution of
$J_G^s$ can be obtained from \cite[Theorem 2.1]{guardo-vtuyl05}. 
We then move on to study some classes of bipartite graphs which
are not complete. We obtain the resolution, syzygies and Betti numbers
of powers of cover ideals of certain bipartite graphs, and as a
by-product, we obtain expression for the regularity of powers of these
ideals.

Section 4 is devoted to the study of resolution and regularity of
powers of cover ideals of certain complete multipartite graphs.  When
$G$ is the cycle of length three or the complete graph on $4$
vertices, we describe the graded minimal free resolution of $J_G^s$
for all $s \geq 1$. This allow us to compute the Betti numbers,
Hilbert series and the regularity of $J_G^s$ for all $s \geq 1$. As a
consequence, for cover ideals of complete tripartite and $4$-partite
graphs, we obtain precise expressions for the Betti numbers and the
regularity of $J_G^s$.  We conclude our article with a conjecture on
the resolution of $J_G^s$ for all $s \geq 1$, where $G$ is a complete
multipartite graph. 

\vskip 2mm
\noindent
\textbf{Acknowledgements:} Part of the work was done while the second
author was visiting Indian Institute of Technology Madras. He would
like to thank IIT Madras for their hospitality during the visit. All
our computations were done using Macaulay 2, \cite{M2}. We would like
to thank Huy T\`ai H\`a and S. A. Seyed Fakhari for their comments
which helped us improve the exposition.

\section{Preliminaries}

In this section, we set the notation 
for the rest of the paper. All the graphs that we consider in this
article are finite, simple and without isolated vertices. For a graph
$G$, $V(G)$ denotes the set of all vertices of $G$ and $E(G)$ denotes
the set of all edges of $G$. A graph $G$ is said to be a \textit{complete
multipartite} graph if $V(G)$ can be partitioned into sets $V_1,
\ldots, V_k$ for some $k \geq 2$ such that $\{x,y\} \in E(G)$ if and
only if $x \in V_i$ and $y \in V_j$ for $i \neq j$. When $k = 2$, 
the graph is called a \textit{complete bipartite} graph. If $k = 2$
with $|V_1| = m$ and $|V_2| = n$, we denote the corresponding complete
bipartite graph by $K_{m,n}$ or by $K_{V_1,V_2}$. If $G$ and $H$ are
graphs, then $G \cup H$ denote the graph on the vertex set $V(G) \cup
V(H)$ with $E(G\cup H) = E(G) \cup E(H)$. A graph $G$
is called a \textit{bipartite} graph if $V(G) = V_1 \sqcup V_2$ such
that $\{x,y\} \in E(G)$ only if $x \in V_1$ and $y \in V_2$. A subset
$w = \{x_{i_1}, \ldots, x_{i_r}\}$ of $V(G)$ is said to be a
\textit{vertex cover} of $G$ if $w \cap e \neq \emptyset$ for every $e
\in E(G)$. A vertex cover is said to be minimal if it is minimal with
respect to inclusion.

Let $G$ be a graph with $V(G) = \{x_1, \ldots, x_n\}$. Let $K$ be a
field and $R = K[x_1,\ldots,x_n]$. The \textit{edge ideal} of $G$ is
defined to be $I(G) = \langle\{x_ix_j ~
: ~ \{x_i,x_j\} \in E(G)\} \rangle \subset S$  and the \textit{cover
ideal} of $G$ is defined to be $J_G = \langle
\{x_{i_1}\cdots x_{i_r} ~ : ~ \{x_{i_1}, \ldots, x_{i_r}\}$. 
It can also be seen that $J_G$ is the Alexander dual of
$I(G)$.

Let $S=R/I$, where $R$ is a polynomial ring over $K$ and $I$ a
homogeneous ideal of $R$. For a finitely generated graded $S$-module
$M=\oplus M_i$, set
\[
 t_i^{S}(M)=\max \{j : \Tor_{i}^{S}(M,K)_j \neq 0 \},
\]
with $t_i^{S}(M)=-\infty $ if $\Tor_i^S(M,K) = 0.$ The
\emph{Castelnuovo-Mumford regularity}, denoted by $\reg_S(M)$, of an $S$-module
$M$ is defined to be
\[
 \reg_SM= \max \{t_i^{S}(M)-i: i \geq 0\}.
\]

\vskip 3mm

\section{Bipartite Graphs}
In this section, we study the regularity of powers of cover ideals of
certain bipartite graphs. 
We begin with a simple observation concerning the vertex covers of a
bipartite graph.

\begin{proposition}
Let $G$ be a bipartite graph on $n+m$ vertices. Then $G$ is a complete
bipartite graph if and only if $J_G$ is generated by a regular sequence.
\end{proposition}

\begin{proof}Let $V(G) = X \sqcup Y$ be the partition of the vertex 
set of $G$ with $X = \{x_1,\ldots,x_n\}$ and $Y =
\{y_{n+1},\ldots,y_{n+m}\}$. First note that $J_G$ is generated by a
regular sequence if and only if for any two minimal vertex covers $w,
w'$, $w \cap w' = \emptyset$. If $G = K_{n,m}$, then $J_G = (x_1\cdots
x_n, y_{n+1}\cdots y_{n+m})$ which is a regular sequence. Conversely,
suppose $G$ is not a complete bipartite graph.  Since $G$ is a
bipartite graph, note that $\prod_{x_i \in X} x_i, \prod_{y_j \in
Y}y_j \in J_G$ are minimal generators of $J_G$.  Therefore, there
exist $x_{i_0} \in X$ and $y_{i_0} \in Y$ such that $\{x_{i_0},
y_{i_0}\} \notin E(G)$. Then $w = \{x_i, y_j ~ : ~ i \neq i_0 \text{
and }y_j \in N_G(x_{i_0})\}$ is a minimal vertex cover of $G$ that
intersects $X$ as well as $Y$ non-trivially. Therefore $J_G$ is not a
complete intersection.
\end{proof}

First we discuss the regularity of powers of cover ideals of complete
bipartite graphs. Since the cover ideal of a complete bipartite graph
is a complete intersection, the result is a consequence of
\cite[Theorem 2.1]{guardo-vtuyl05}.

\begin{theorem}\label{com.bipar00}
Let $J = J_{K_{m,n}}$ be the cover ideal of the complete
bipartite graph $K_{m,n}, ~ m \leq n$. Then $\reg(J^s) = sn + m -1$
for all $s \geq 1$.
\end{theorem}

\begin{proof}
Consider the ideal $I = (T_1, T_2) \subset R = K[T_1, T_2]$ with $\deg T_1
= m$ and $\deg T_2 = n$. It
follows from \cite[Theorem 2.1]{guardo-vtuyl05} that the resolution of
$I^s$ is
\begin{eqnarray}\label{ci-res}
  0 \to \underset{a_i \geq 1}{\bigoplus_{a_1+a_2 = s+1}}R(-a_1m-a_2n) \to
  \bigoplus_{a_1+a_2 = s} R(-a_1m-a_2n) \to I^s \to 0.
\end{eqnarray}

Note that $J = (x_1\cdots x_m, y_{m+1}\cdots y_{m+n})$. Set
$x_1\cdots x_m = T_1$ and $y_{m+1}\cdots y_{m+n} = T_2$. Then $J^s$
has the minimal free resolution as in (\ref{ci-res}). If $m \leq n$,
then $\reg(J^s) = sn + m - 1$.
%
%
%and $n$ respectively, $\reg(J) = m+n-1$.
%
%By induction, assume that $\reg(J^{s-1}) = (s-1)n + m-1$.
%Consider the exact sequence
%$$
%0 \longrightarrow \frac{S}{J^s : t_n}(-n) \overset{\cdot t_n}
%\longrightarrow \frac{S}{J^s} \longrightarrow \frac{S}{(J^s,
%t_n)} \longrightarrow 0.
%$$
%
%Note first that $(J^s, t_n) = (t_m^s, t_n)$. Therefore, $\reg(J^s,
%t_n) = sm + n -1$.  
%Also, $J^s : t_n$ is generated by
%$\displaystyle{\left\{\frac{t_m^it_n^j}{\gcd(t_n, t_m^it_n^j)}
%~ : ~ i+j = s\right\}}$. Therefore, $J^s : t_n = J^{s-1}$.
%Hence by induction, $\reg(J^s : t_n) = \reg(J^{s-1}) = (s-1)(n) +
%m - 1$. If $m < n$, then $\reg(S/J^s, t_n) + 1 < \reg(S/J^{s-1}(-n))$
%and therefore it follows that $\reg(S/J^s) = 
%\reg(S/J^{s-1}) + n.$ Hence, $\reg(J^s) = sn + m -1$.
\end{proof}

It follows from Theorem \ref{com.bipar00} that in the case of the
complete bipartite graph $K_{m,n}$, the stability index is $1$ and the constant
term is $\tau-1$, where $\tau$ is the size of a minimum vertex cover.

If the graph is not a complete bipartite graph, then the cover ideal
is no longer a complete intersection. If $G$ is a Cohen-Macaulay
bipartite graph, then it was shown by F. Mohammadi and S. Moradi that
the vertex cover ideals are weakly polymatroidal. Therefore, they have
linear quotients and hence all the powers have linear resolution,
\cite[Theorem 2.2]{mohammadi-moradi10}. It would be quite a
challenging task to obtain the syzygies and Betti numbers of powers of
cover ideals of all bipartite graphs. Therefore, we restrict our
attention to some structured subclasses of bipartite graphs.

\noindent
\begin{minipage}{\linewidth}
\begin{minipage}{0.25\linewidth}
  \begin{figure}[H]
\begin{tikzpicture}
\draw (3,7)-- (1,7);
\draw (1,7)-- (3,6.5);
\draw (1,7)-- (3,6);
\draw (1,7)-- (3,5.5);
\draw (1,7)-- (3,5);
\draw (1,7)-- (3,4.5);
\draw (1,6.5)-- (3,7);
\draw (1,6.5)-- (3,6.5);
\draw (1,6.5)-- (3,6);
\draw (1,6.5)-- (3,5.5);
\draw (1,6.5)-- (3,5);
\draw (1,6.5)-- (3,4.5);
\draw (1,6)-- (3,5.5);
\draw (1,6)-- (3,5);
\draw (1,6)-- (3,4.5);
\draw (1,5.5)-- (3,5);
\draw (1,5.5)-- (3,4.5);
\draw (1,6)-- (3,7);
\draw (1,6)-- (3,6.5);
\draw (1,6)-- (3,6);
\draw (1,5)-- (3,5);
\draw (1,4.5)-- (3,5);
\draw (1,4)-- (3,5);
\draw (1,5)-- (3,4.5);
\draw (1,4.5)-- (3,4.5);
\draw (1,4)-- (3,4.5);
\begin{scriptsize}
\fill (1,7) circle (1.5pt);
\draw(1.14,7.13) node {$x_1$};
\fill (1,6.5) circle (1.5pt);
\draw(1.14,6.63) node {$x_2$};
\fill (1,6) circle (1.5pt);
\draw(1.14,6.14) node {$x_3$};
\fill (1,5.5) circle (1.5pt);
\draw(1.14,5.64) node {$x_4$};
\fill (1,5) circle (1.5pt);
\draw(1.14,5.13) node {$x_5$};
\fill (1,4.5) circle (1.5pt);
\draw(1.14,4.63) node {$x_6$};
\fill (1,4) circle (1.5pt);
\draw(1.14,4.14) node {$x_7$};
\fill (3,7) circle (1.5pt);
\draw(3.14,7.13) node {$y_1$};
\fill (3,6.5) circle (1.5pt);
\draw(3.14,6.63) node {$y_2$};
\fill (3,6) circle (1.5pt);
\draw(3.14,6.14) node {$y_3$};
\fill (3,5.5) circle (1.5pt);
\draw(3.14,5.64) node {$y_4$};
\fill (3,5) circle (1.5pt);
\draw(3.14,5.13) node {$y_5$};
\fill (3,4.5) circle (1.5pt);
\draw(3.14,4.63) node {$y_6$};
\end{scriptsize}
\end{tikzpicture}
\caption*{$K_{U_1,V} \cup K_{U_2,V_2}$}
\end{figure}
\end{minipage}
\begin{minipage}{0.73\linewidth}
  \begin{notation}
	For disjoint vertex sets $U$ and $V$, let $K_{U,V}$ denote the
	complete bipartite graph on $U \sqcup V$. Let $U_1 = \{x_1,
	\ldots, x_k\}, ~U_2 = \{x_{k+1}, \ldots, x_n\}, ~V_1 =
	  \{y_1, \ldots, y_r\}$ and $V_2 = \{y_{r+1}, \ldots, y_m\}$. Let
	  $U = U_1 \cup U_2$ and $V = V_1\cup V_2$. In the following, we
	  consider the bipartite graph $G$ on the vertex set $U \sqcup V$
	  with edges $E(G) = E(K_{U_1,V}) \cup E(K_{U_2,V_2})$. Note that
	  while the vertex sets of $K_{U_1,V}$ and $K_{U_2,V_2}$ are not
	  disjoint, the edge sets are. The figure on the left is an
	  example of such a graph with $U_1 = \{x_1,x_2,x_3\}, U_2 =
	  \{x_4,x_5,x_6,x_7\}, V_1 = \{y_1,y_2,y_3,y_4\}$ and $V_2 =
	  \{y_5,y_6,y_7\}$.
  \end{notation}
\end{minipage}
\end{minipage}

%Let $U_1, \ldots, U_n$ and $V_1, \ldots, V_n$ be sets of vertices. Set
%$U = U_1 \cup \cdots \cup U_n$ and $V = V_1 \cup \cdots \cup V_n$.
%Let $G = K_{U_1,V} \cup K_{U_2,V_2} \cup \cdots \cup K_{U_n,V_n}$.
%For $1 \leq i, j \leq n$, set $U_i = \{x_{i_1},\ldots, x_{i_{n_i}}\}$
%and $V_j = \{y_{j_1},\ldots, y_{j_{m_j}}\}$. Let $J_G$ denote the
%cover ideal of $G$. Then $J_G$ can be identified with the
%cover ideal described in Theorem \ref{thm:gen-ncb} by taking
%$x_i = \prod_{x_{i_r} \in U_i} x_{i_r}$ and $y_j = \prod_{y_{j_s} \in
%V_j} y_{j_s}$. But in this case, since $\deg x_i$ is not necessarily
%one, the resolution need not be linear and hence the computation of
%regularity is non-trivial. We now study the resolution and the
%regularity for $n = 2$.

\vskip 2mm \noindent
\begin{theorem}\label{thm:bipar02}
Let $U = U_1 \sqcup U_2$ and $V = V_1 \sqcup V_2$ be a collection of
vertices with $|U| = n, ~|U_i| = n_i, ~ |V| = m, ~|V_i| = m_i$ and $1
\leq n_i, m_i$ for $i = 1, 2$. Let $G$ be the bipartite graph $K_{U_1, V}
\cup K_{U_2, V_2}$.  Let $R = K[x_1, \ldots, x_n, y_1, \ldots, y_m]$.
Let $J_G \subset R$ denote the cover ideal of $G$. Then the
graded minimal free resolution of $R/J_G$ is of the form:

\[
  0 \longrightarrow  R(-n-m_2) \oplus R(-m-n_1)
  \longrightarrow  R(- (n_1+m_2)) 
\oplus R(-m) \oplus  R(-n)   \longrightarrow  R \longrightarrow  0.
\]
In particular,
%$$ 
%\reg(J_G) = \max \{ k_2 + u_1 + u_2 -1, u_1 + k_1 + k_2 -1     \}
%$$
%We may also write it as 
$$ 
\reg(J_G) = \max \{ n+ m_2-1, m+ n_1 -1\}
$$
\end{theorem}

\begin{proof}
It can easily be seen that 
the cover ideal $J_G$ is generated by $g_1 = x_1 \cdots x_n,\;
g_2 = y_1 \cdots y_m,$ and $g_3 = x_1\cdots x_{n_1} y_{m_1+1}\cdots
y_m$. Set $X_1 = x_1 \cdots x_{n_1}; ~X_2 = x_{n_1+1}\cdots x_n;
~ Y_1 = y_1 \cdots y_{m_1}$ and $Y_2 = y_{m_1+1}\cdots y_m$. Then we can write
$g_1 = X_1X_2, ~g_2 = Y_1Y_2$ and $g_3 = X_1Y_2$.

Consider the minimal graded free resolution of $R/J_G$ over $R$:
\[
   \cdots \longrightarrow F 
  \overset{\partial_2}{\longrightarrow}   R(-n) 
\oplus R(-m) \oplus R(- (n_1+m_2))   \overset{\partial_1}{\longrightarrow} R \longrightarrow  0,
\]
where 
 $ \partial_1(e_1)=  g_1$, $ \partial_1(e_2)=  g_2$, 
and $ \partial_1(e_3)= g_3$. 

Let $ae_1+be_2+ce_3 \in \ker \partial_1$. Then $aX_1X_2+bY_1Y_2+cX_1Y_2 =
0$. Solving the above equation, it can be seen that, 
$$\ker \partial_1 = \span_R\{Y_2e_1 -
X_2e_3, X_1e_2 - Y_1e_3\}.$$
Also, it is easily verified that these two generators are $R$-linearly
independent. Hence $\ker \partial_1 \cong R^2$. Note that $\deg
(Y_1e_1 - X_2e_3) = n+m_1$ and $\deg(X_1e_2 - Y_1e_3) = n_1 +
m$. Therefore, we get the minimal free resolution of $R/J_G$ as:
\[
  0 \longrightarrow  R(-n-m_2) \oplus R(-m-n_1)
  \overset{\partial_2}{\longrightarrow}  R(- (n_1+m_2)) 
\oplus R(-m) \oplus  R(-n)   \overset{\partial_1}{\longrightarrow}  R \longrightarrow  0,
\]
where $\partial_2(a,b) = a(Y_1e_1 - X_2e_3) + b(X_1e_2 - Y_1e_3)$.
The regularity assertion follows immediately from the resolution.
\end{proof}

Our aim is to compute the syzygies and the Betti numbers of $J_G^s$,
where $J_G$ is the cover ideal discussed in Theorem \ref{thm:bipar02}.
For this purpose, we first study the resolution of powers of the ideal
$(X_1X_2,X_1Y_2, Y_1Y_2)$ and obtain the resolution and regularity of
the cover ideal as a consequence. It is known from \cite{hhz04} that
all the powers of this ideal has a linear resolution. We explicitly
compute the syzygies and Betti numbers for the powers of this ideal.

\begin{theorem}\label{thm:two-bipartite-graph} Let
$R=K[X_1,X_2,Y_1,Y_2]$ and
$J=(X_1X_2,X_1Y_2,Y_1Y_2)$ be an ideal of $R$. Then, for $s \geq 2$,
the minimal free resolution of $R/J^s$ is of the form
\[
  0 \longrightarrow R^{s\choose 2}  \longrightarrow R^{2{s+1\choose 2}}
 \longrightarrow R^{s+2\choose 2} \longrightarrow R \longrightarrow 0,
\]

and $\reg(J^s) = 2s$. 
\end{theorem}
\begin{proof}
Denote the generators of $J$ by $g_1=X_1X_2$, $g_2=X_1Y_2$, and
$g_3=Y_1Y_2$. Note that $J$ is the edge
ideal of $P_4$, the path graph on the vertices $\{X_1,X_2,Y_1,Y_2\}$. Since $P_4$ is
chordal, it follows that $R/J^s$ has a linear resolution for all $s
\geq 1$, \cite{hhz04}. Now we compute the Betti number of the $R/J^s$.
Write
\[
 (g_1, g_2, g_3)^s= ( g_1^s, g_1^{s-1}(g_2,g_3), g_1^{s-2}(g_2,g_3)^{2}, \dots , g_1(g_2,g_3)^{s-1},  (g_2,g_3)^{s})
\]
where 
\[
 (g_2,g_3)^{t}= (g_2^t,  g_2^{t-1}g_3 , g_2^{t-2}g_3^2, \dots,  g_2 g_3^{t-1}, g_3^t).
\]
For $i \geq j$, set $M_{i,j} = g_1^{s-i}g_2^{i-j}g_3^j = 
(X_1X_2)^{s-i}Y_2^iY_1^jX_1^{i-j}.$ 
It follows that $\mu(J^s) = \frac{(s+1)(s+2)}{2}$. 

Set $\beta_1 = \frac{(s+1)(s+2)}{2}$. Let $\{e_{p,q} \mid 0 \leq p
\leq s; 0 \leq q \leq p\}$ denote the standard basis for
$R^{\beta_1}$. Let $\partial_1 : R^{\beta_1} \longrightarrow R$ be the
map $\partial_1(e_{p,q}) = M_{p,q}$.
Since $g_i$'s are monomials, the kernel is generated by binomials of
the form $m_{i,j}M_{i,j} - m_{k,l}M_{k,l}$, where $m_{p,q}$'s are
monomials in $R$.  Since the resolution of $R/J^s$ is linear, it is
enough to find the linear syzygy relations among the generators of
$\ker(\partial_1)$. To find these linear syzygies, we need to find
conditions on $i,j,k,l$ such that $\frac{M_{i,j}}{M_{k,l}}$ is equal
to $\frac{X_p}{Y_q}$ or $\frac{Y_q}{X_p}$ for some $p, q$. First of
all, note that for such linear syzygies, $|i-k|, |j-l|
\leq 1$. If $i = k$ and $j = l+1$, then
$\displaystyle{\frac{M_{i,j}}{M_{i,j+1}} = \frac{g_3}{g_2} = 
\frac{Y_1}{X_1}}$. We get the same relation if $i = k$ and $j = l-1$.
If $i = k+1$ and $j = l$, then
$\displaystyle{\frac{M_{i,j}}{M_{i-1,j}} = \frac{g_2}{g_1} = 
\frac{Y_2}{X_2}}$. As before, $i = k-1$ yields the same relation.
Therefore, the kernel is minimally generated by
\[
\left\{Y_1 e_{i,j} - X_1 e_{i,j+1}, X_2 e_{i,j} - Y_2 e_{i-1,j} \mid
0 \leq j < i \leq s \right\}.
\]
Hence, $\mu(\ker \partial_1) = 2{s+1 \choose 2}$. Write the basis elements
of $R^{2{s+1\choose 2}}$ as 
\[
  \{e_{1,i,p}, e_{2,i,q} \mid 1 \leq i \leq s, 0 \leq p < i \text{
  and } 0 \leq q < i\}
\]
and define
$\partial_2 : R^{2{s+1\choose 2}} \longrightarrow R^{\beta_1}$ by
\begin{eqnarray*}
\partial_2(e_{1,i,p}) & = & Y_1e_{i,p} - X_1e_{i,p+1} \\
\partial_2(e_{2,i,q}) & = & X_2e_{i,q} - Y_2e_{i-1,q}.
\end{eqnarray*}

By \cite[Proposition 3.2]{morey10}, $\pdim(R/J^s) = 3$ for all $s \geq
2$.
Hence we conclude that the minimal graded free resolution of
$R/J^s$ is of the form
\[
  0 \longrightarrow R^{\beta_3}  \longrightarrow R^{2{s+1 \choose 2}}
 \overset{\partial_2}{\longrightarrow}
 R^{{s+2 \choose 2}}
 \overset{\partial_1}{\longrightarrow} R \longrightarrow
 0.
\]
Therefore, 
\[
  \beta_3 - 2{s+1 \choose 2} + {s+2\choose2} - 1=0,
\]
so that $\beta_3={s \choose 2}$. Now we compute the generators of
the second syzygy. Again, since the resolution is linear, it is enough
to compute linear generators. First, note that 
\[
  B = \{Y_2e_{1,i,j} - X_2 e_{1,i+1,j} + Y_1 e_{2,i+1,j} - X_1
e_{2,i+1,j+1} \mid 0 \leq j < i < s \} \subseteq \ker \partial_2.
\]
It can easily be verified that $B$ is $R$-linearly independent. Since $\mu(\ker
\partial_2) = {s \choose 2}$, $B$ generates $\ker \partial_2$.
Let $\{E_{i,j} \mid 0 \leq j < i < s\}$ denote the standard basis for
$R^{s \choose 2}$. Define $\partial_3 : R^{s \choose 2} \to
R^{s(s+1)}$ by 
\[
  \partial_3(E_{i,j}) = Y_2e_{1,i,j} - X_2
e_{1,i+1,j} + X_1 e_{2,i+1,j} - Y_1 e_{2,i+1,j+1}.
\]
Therefore, we get the minimal free resolution of $R/J^s$ as
\[
  0 \rightarrow R^{s \choose 2}
  \xrightarrow{\partial_3} R^{2{s+1\choose 2}}  \xrightarrow{\partial_2} 
  R^{s+2\choose 2}  \xrightarrow{\partial_1} R
  \longrightarrow 0.
\]
Since the resolution is linear, $\reg J^s = 2s$.
\end{proof}
As an immediate application of the previous theorem, we obtain
resolution and regularity of powers of the cover ideals of graphs
discussed in Theorem \ref{thm:bipar02}.
%As an immediate consequence, we obtain the resolution and the
%regularity of a specific class of bipartite graphs.
\begin{theorem}\label{cor:bipar02}
Let $U = U_1 \sqcup U_2$ and $V = V_1 \sqcup V_2$ be a collection of
vertices with $|U| = n, ~|U_i| = n_i, ~ |V| = m, ~|V_i| = m_i$ and $1
\leq n_i, m_i$ for $i = 1, 2$. Let $G$ be the bipartite graph $K_{U_1, V}
\cup K_{U_2, V_2}$.  Let $R = K[x_1, \ldots, x_n, y_1, \ldots, y_m]$.
Let $J_G \subset R$ denote the cover ideal of $G$. Then the
minimal free resolution of $R/J_G^s$ is of the form:
\[
  0 \longrightarrow R^{s\choose 2}  \longrightarrow R^{2{s+1\choose 2}}
 \longrightarrow R^{s+2\choose 2} \longrightarrow R \longrightarrow 0.
  \]
Moreover,
\[
  \reg J_G^s =  \max \left\{
	\begin{array}{ll} 
	  (s-j)n_1 + (s-i)n_2 + jm_1 + im_2 & \text{ for } 0 \leq j \leq i \leq s\\
	  (s-j)n_1 + (s-i)n_2 + (j+1)m_1 + im_2 - 1& \text{ for } 0 \leq j < i \leq s\\
	  (s-j)n_1 + (s-i+1)n_2 + jm_1 + im_2 - 1& \text{ for } 0 \leq j < i \leq s\\
	  (s-j)n_1 + (s-i)n_2 + (j+1)m_1 + (i+1)m_2 - 2& \text{ for } 0 \leq j < i < s
	\end{array}
  \right. .
\]
\end{theorem}

\begin{proof}
Let $R = K[x_1,\ldots,x_n,y_1,\ldots,y_m]$. Then $J_G = (x_1\cdots
x_n, y_1\cdots y_m, x_1\cdots x_{n_1}y_{m_1+1}\cdots y_m)$. Set $X_1
= x_1\cdots x_{n_1}, X_2 = x_{n_1+1}\cdots x_n, Y_1 = y_1\cdots
y_{m_1}$ and $Y_2 = y_{m_1+1}\cdots y_m$. Then $J_G = (X_1X_2, Y_1Y_2,
X_1Y_2)$. Therefore, it follows from Theorem \ref{thm:bipar02} that
$J_G^s$ has the given minimal free resolution.

To compute the regularity, we need to obtain the degrees of the
syzygies. Following the notation in the proof of
Theorem \ref{thm:bipar02}, we can see that 
\begin{eqnarray*}
  \deg e_{i,j} & = & (s-j)n_1 + (s-i)n_2 + jm_1 + im_2,  \\
  \deg e_{1,i,j} & = & (s-j)n_1 + (s-i)n_2 + (j+1)m_1 + im_2,  \\
  \deg e_{2,i,j} & = & (s-j)n_1 + (s-i+1)n_2 + jm_1 + im_2, \\
  \deg E_{i,j} & = & (s-j)n_1 + (s-i)n_2 + (j+1)m_1 + (i+1)m_2.
\end{eqnarray*}
Therefore, the assertion on the regularity follows.
\end{proof}

\subsection{Discussion}\label{disc}
It has been proved by Hang and Trung, \cite{ht2017}, that if $G$ is a
bipartite graph on $n$ vertices and $J_G$ is the cover ideal of
$G$, then there exists a non-negative integer $e$ such that for $s
\geq n+2 $, $\reg(J_G^s) = \deg(J_G) s + e$, where $\deg(J_G)$ denote
the maximal degree of minimal monomial generators of $J_G$. It follows
from Theorem \ref{com.bipar00} that if $G = K_{n,m}$ with $n \geq m$,
then $e = m-1$ and the index of stability is $1$. If the graph is not
a complete bipartite graph, then $e$ does not uniformly 
represent a combinatorial invariant associated to the graph as can be
seen in the computations below. We compute
the polynomial $\reg(J_G^s)$ for some classes of bipartite graphs that
are considered in Corollary \ref{cor:bipar02}. We see that $e$
depends on the relation between the integers $n_1, n_2, m_1$ and
$m_2$. Just to illustrate the computation of the polynomial from
Theorem \ref{cor:bipar02}, we compute $\reg(J_G^s)$ in some cases
below:
%To find the polynomial $\reg(J_G^s)$, we need to identify which
%of those relations give rise to the regularity. If we assume that all
%those integers are at least one, then $\deg e_{i,j}$ will not play a
%role in deciding the regularity. Note that $\deg e_{1,i,j} - \deg
%e_{2,i,j} = m_1 - n_2$ and $\deg e_{2,i,j} - \deg E_{i,j} = n_2 -
%(m_1+m_2)$. Therefore, depending on the relation between these
%integers, one can decide the expression for the regularity of $J_G^s$.
%We first deal with two simple cases. The first case shows that the
%constant term can as well be zero.
\begin{enumerate}
  \item If $m_1 = m_2  = 1$, then it can be seen that for $s \geq 2$,
	\[\reg(J_G^s) = \max\{(s-j)n_1 + (s-i+1)n_2 + jm_1 + i m_2 -
	1 ~ : ~ 0 \leq j < i \leq s\}.\]
	Since $(s-j)n_1 + (s-i+1)n_2 + jm_1 + i m_2 - 1 = s(n_1+n_2) +
	j(1-n_1) + i(1-n_2) + n_2-1$ and $n_1 > 1, ~ n_2
	> 1$, this expression attains maximum when $i$ and $j$ attain minimum,
	i.e., if $j = 0$ and $i = 1$. Therefore, $\reg(J_G^s) = n s$.
	Thus, in this case, $e = 0$. It can also be noted that, since $n
	\geq 2$, it follows from Theorem \ref{thm:bipar02} that $\reg(J_G)
	= n$. Therefore,  in this case, the stability index is also equal to $1$.
  \item If $n_1 = n_2 = m_1 = m_2 = \ell > 1$, then  for $s \geq 2$,
	\[\reg(J_G^s) = \max\{(s-j)n_1 + (s-i)n_2 + (j+1)m_1 + (i+1) m_2 -
	2 ~ : ~ 0 \leq j < i \leq s\}.\]
	Therefore, $\reg(J_G^s) = 2\ell s + (2\ell-2)$ and hence $e =
	2\ell - 2$. Note that in this case, the stability index is $2$. 
%\end{enumerate}
%
%Now we discuss some more cases and compute the polynomial
%$\reg(J_G^s)$, $s \geq 2$. Assume that $n_2 = m_1$.
%\begin{enumerate}
  \item $n_1 \geq m_2 \geq n_2=m_1$: Note that, in this case,
	$\deg(J_G) = n_1+m_2$.  We have
	\[\reg(J_G^s) = \max\{
	(s-j)n_1+(s-i)n_2 + (j+1)m_1 + (i+1)m_2 -2~ : ~ 0 \leq j < i < s\}.\]
	Since $(s-j)n_1 + (s-i)n_2 + (j+1)m_1 + (i+1)m_2 -2 = s(n_1+n_2) +
	j(m_1-n_1) + i(m_2-n_2) + m_1+m_2-2$. Since $n_1 \geq m_1$ and
	$m_2 \geq n_2$, the above expression attains the maximum when $i$
	attains the maximum and $j$ attains the minimum, i.e., if $i = s-1$ and $j =
	0$. Therefore, $\reg(J_G^s) = (n_1+m_2)s + (n_2+m_1-2)$ and hence
	$e = n_2+m_1-2$.

  \item $n_1 \geq n_2 = m_1 \geq m_2$:  In this case, $\deg(J_G) = n$
	and
	\[\reg(J_G^s) = \max\{
	(s-j)n_1+(s-i)n_2 + (j+1)m_1 + (i+1)m_2 -2~ : ~ 0 \leq j < i < s\}.\]
	As in the previous case, one can conclude that the maximum is
	attained when $i = 1$ and $j = 0$. Therefore, $\reg(J_G^s) = n s
	+ (2m_2 - 2)$. Thus $e = 2m_2-2$.
	
  \item $n_2 = m_1 \geq n_1 \geq m_2$: As done earlier, one can
	conclude that  $\reg(J_G)^s = n s + (2m_2 - 2)$ and
	hence $e = 2m_2-2$.

%  \item Similarly analyzing the relation between these integers, one can
%	obtain the regularity expressions as follows:
%
%	\begin{tabular}{|l|l|}
%	  \hline
%	  $\mathbf{n_2 = m_1}$ & $\mathbf{\reg(J_G^s)}, ~ ~ s \geq 2$  \\
%	  \hline
%	  \hline
%	  $n_2 = m_1 \leq n_1 \leq m_2$ & $(n_1+m_2) s + (m_1+n_2-2)$ \\
%	  \hline
%	  $n_1 \leq n_2=m_1 \leq m_2$ & $(m_1+m_2)s + (n_1+m_2-2)$ \\
%	  \hline
%	  $n_1 \leq m_2\leq n_2 = m_1$ & $(m_1+m_2)s + (n_1+m_2-2)$\\
%	  \hline
%	\end{tabular}
%
\end{enumerate}

%\vskip 2mm \noindent
%It can be noted that in the case $n_2 = m_1$, the stability index 
%is $2$. 
\begin{remark}
If $G = K_{U_1,V} \cup K_{U_2,V_2} \cup K_{U_3,V_3}$, for some set
of vertices $U_i, V_i$, then one can still describe the complete
resolution and the regularity of $J_G^s$ using a similar approach.
However, the resulting syzygies are not so easy to describe though the
generating sets are similar.
Therefore, we restrict ourselves to the above discussion.
\end{remark}

\section{Complete Multipartite Graphs}

In this section our goal is to understand the resolution and
regularity of the powers of cover ideals of complete $m$-partite
graphs. Let $G$ be a complete $m$-partite graph and let $J_G$ be the
cover ideal of $G$. 
Let $R=K[x_1,\ldots,x_m]$. It is known that the cover ideal
$J_G$ of complete graph $G=K_m$ is generated by all squarefree
monomials $x_1x_2\cdots \hat{x_i} \cdots x_m$ of degree $m-1$.
Moreover one can also identify this cover ideal with the
squarefree Veronese ideal $I=I_{m,m-1}$, and hence it
is a polymatroidal ideal, \cite{hh05}. 
%
% 
%Let $I \subset R$ be an ideal of $R$. The ideal $I$ is said to have
%linear quotients if there exists an ordered system of homogeneous
%generators $f_1,\ldots,f_{\ell}$ of $I$ such that for all
%$j=1,\ldots,\ell$ the colon ideals $(f_1,\ldots,f_{j-1}):f_j$ are
%generated by a subset of $\{ x_1,\ldots,x_n\}$.
Therefore by the results of Conca-Herzog, \cite{ch03}, and
Herzog-Hibi, \cite{Herzog'sBook}, we have
\iffalse
It was proved by Conca and Herzog that a polymatroidal ideal has
linear  quotients with respect to the reverse lexicographical order of
the generators \cite[Proposition 5.2]{ch03}. They further proved that
the product of polymatroidal monomial ideals are polymatroidal
\cite[Theorem 5.3]{ch03}.  Herzog and Hibi proved that if $I$ is
generated in degree $d$ and has linear quotients, then $I$ has a
$d$-linear resolution, \cite[Proposition 8.2.1]{Herzog'sBook}.  
Therefore we have the following: 
\fi
 
\begin{remark}\label{rem:lin res powers} The cover ideal $J_G$ of
the complete graph $G=K_m$ has linear quotients and hence has linear
resolution. Moreover $J_G^s$ has linear resolution for all $s \geq
1$.
\end{remark}

If $I$ is an ideal of $R$ all of whose powers have linear resolution,
then depth $R/I^k$ is a non-increasing function of $k$ and $\depth
R/I^k$ is constant for all $k \gg 0$, \cite[Proposition 2.1]{hh05}. 
Further, we have:
 
\begin{remark}\cite[Corollary 3.4]{hh05} \label{rem:depth powers} Let
$R=K[x_1,\cdots,x_m]$ and $J_G$ be the cover ideal of $G = K_m$. Then 
\[
 \depth R/{J_G^s} = \max \{  0, m-s-1  \}.
\]
In particular, $\depth R/{J_G^s} =0$ for all $s \geq m-1$.
\end{remark}
%We begin our investigation by studying the resolution of powers of
%cover ideal of $K_3$.

\subsection{Complete tripartite:}
%It has been shown by
%Beyarslan et al., \cite[Theorem 5.2]{selvi_ha}, that $\reg(I(C_3)^s) =
%2s$ for all $s \geq 2$. 
We first describe the graded minimal free
resolution of $J_{K_3}^s$ for all $s \geq 1$. We also obtain the
Hilbert series of the powers.

\begin{theorem}\label{C3}
Let $R = K[x_1, x_2, x_3]$ and $I = (x_1x_2, x_1x_3, x_2x_3)$. Then
the graded minimal free resolution of $R/I$ is of the form:
\[ 0 \to R(-3)^2 \to R(-2)^3 \to R  \to 0. \]
For $s \geq 2,$ the graded minimal free resolution of $R/I^s$ is of
the form:
\[0 \to R(-2s-2)^{s \choose 2} \to R(-2s-1)^{2{s+1\choose 2}} \to
  R(-2s)^{s+2\choose 2} \to R \to 0 
\]
so that $\reg_R (I^s)=2s$. Moreover, the Hilbert series of $R/I^s$ is given by 
\begin{eqnarray*}
 H(R/I^s,t) = \frac{1+2t+3t^2+ \cdots +2st^{2s-1} - \left({s +2 \choose 2}-2s-1  \right)t^{2s} }{(1-t)}.
\end{eqnarray*}
\end{theorem}
\begin{proof}
It is clear that the resolution of $I$ is as given in the assertion of
the theorem. Since $I$ is the cover ideal of $K_3$, $I^s$ has linear
minimal free resolution for all $s \geq 1$.
Note also that by \cite[Lemma 3.1]{cms02},
$\depth R/I^s = 0$ for all $s \geq 2$ so that $\pdim R/I^s = 3$
for all $s \geq 2$. Hence, the minimal free resolution of $R/I^s$
is of the form:
\[
  0 \to R(-2s-2)^{\beta_3} \overset{\partial_3}{\longrightarrow}
  R(-2s-1)^{\beta_2}
  \overset{\partial_2}{\longrightarrow} R(-2s)^{\beta_1}
  \overset{\partial_1}{\longrightarrow} R \to 0
\]
Now we describe the syzygies and Betti numbers of $R/I^s$ for
$s \geq 2$.  Let $g_1 = x_1x_2, g_2 = x_1x_3$ and $g_3 = x_2x_3$.
%Write $I^s = (g_1^s, g_1^{s-1}(g_2, g_3), \ldots, g_1(g_2,g_3)^{s-1},
%(g_2,g_3)^s).$ 
The generators of $I^s$ are of the form
$g_1^{\ell_1}g_2^{\ell_2}g_3^{\ell_3}$, where $0 \leq
\ell_1,\ell_2,\ell_3 \leq s$ and $\ell_1 + \ell_2 +\ell_3 =s$. Set
$f_{\ell_1,\ell_2,\ell_3}=g_1^{\ell_1}g_2^{\ell_2}g_3^{\ell_3}
= x_1^{s-\ell_3}x_2^{s- \ell_2}x_3^{s-
\ell_1}$. It is easy to see that $\mu(I^s)= {s+2 \choose 2}$, since
the cardinality of a minimal generating set of $I^s$
is same as the total number of non-negative integral solution of
$\ell_1 + \ell_2 +\ell_3 =s$, which is ${s+2 \choose 2}$. 

Let $\{ e_{\ell_1,\ell_2,\ell_3} ~ : ~ 0 \leq \ell_1,\ell_2,\ell_3
\leq s; \ell_1 + \ell_2 + \ell_3 = s \}$ denote the standard basis for
$R^{s+2 \choose 2}$ and consider the map $\partial_1 : R^{s+2 \choose
2} \to R$ defined by $\partial_1(e_{\ell_1,\ell_2,\ell_3})
=f_{\ell_1,\ell_2,\ell_3}$. 
%We find the minimal generators for $\ker \partial_1$. 
Since $f_{\ell_1,\ell_2,\ell_3}$'s are monomials, the
kernel is generated by binomials of the form
$m_{\ell_1,\ell_2,\ell_3}f_{\ell_1,\ell_2,\ell_3} -
m_{t_1,t_2,t_3}f_{t_1,t_2,t_3}$, where $m_{i,j,k}$'s 
are monomials in $R$. Also, since the minimal free
resolution is linear, the kernel is generated in degree $1$.  Note
that $\displaystyle{\frac{f_{\ell_1,\ell_2,\ell_3}}{f_{t_1,t_2,t_3}} =
x_1^{t_3-\ell_3}x_2^{t_2-\ell_2}x_3^{t_1-\ell_1}}$. Hence, for
$\displaystyle{\frac{f_{\ell_1,\ell_2,\ell_3}}{f_{t_1,t_2,t_3}}}$ to
be a linear fraction, $|t_i - \ell_i| \leq 1$ for $i = 1,2,3$.
%Therefore, to get the linear minimal generators of the first syzygy,
%we need to set $t_i = \ell_i$ for one $i \in \{1,2,3\}$ and keep the
%other difference equal to $1$.

Let $t_3=\ell_3$,  $t_2=\ell_2+1$ and $t_1=\ell_1-1$. The
corresponding linear syzygy relation is
\begin{eqnarray}\label{thm:3-partite-2ndsyz1}
 x_3 \cdot f_{\ell_1,\ell_2,\ell_3} - x_2  \cdot
 f_{\ell_1-1,\ell_2+1,\ell_3} =0,
\end{eqnarray}
where $ 1 \leq \ell_1 \leq s$, and $0 \leq \ell_2,\ell_3 \leq s-1$.
Note that the number of such relations is equal to the number of integral solution
to $(\ell_1-1) + \ell_2 + \ell_3 = s$, i.e., ${s+1 \choose 2}$.
Similarly, if $t_3=\ell_3+1$, 
$t_2=\ell_2-1$ and $t_1=\ell_1$, then we get the corresponding linear
syzygy relation as
\begin{eqnarray}\label{thm:3-partite-2ndsyz2}
 x_2  \cdot f_{\ell_1,\ell_2,\ell_3} - x_1  \cdot
 f_{\ell_1,\ell_2-1,\ell_3+1} =0,
\end{eqnarray}
where $ 0 \leq \ell_1, \ell_3 \leq s-1$, and $1 \leq \ell_2 \leq s$. 
Note that the linear syzygy relation obtained by fixing $\ell_2$ and
taking $|\ell_i - t_i| = 1$ for $i = 1, 2$
\[
 x_3 \cdot f_{\ell_1,\ell_2,\ell_3} -x_1  \cdot f_{\ell_1-1,\ell_2,\ell_3+1} =0
\]
can be obtained 
from Equations $(\ref{thm:3-partite-2ndsyz1}),\;
(\ref{thm:3-partite-2ndsyz2})$ by setting the $\ell_i$'s
appropriately. Therefore, 
\[
 \ker \partial_1 =
\langle x_3 \cdot e_{\ell_1,\ell_2,\ell_3} - x_2  \cdot
e_{\ell_1-1,\ell_2+1,\ell_3}, ~ x_2  \cdot
e_{\ell_1-1,\ell_2+1,\ell_3} - x_1  \cdot e_{\ell_1-1,\ell_2,\ell_3+1}
~ : ~ 
1 \leq \ell_1 \leq s, \;0 \leq \ell_2,\ell_3 \leq s-1
 \rangle.
\]
Since there are ${s+1 \choose 2}$ minimal generators of each type
(\ref{thm:3-partite-2ndsyz1}) and (\ref{thm:3-partite-2ndsyz2}),
$\mu(\ker\partial_1) = \beta_2  = 2{s+1 \choose 2}$. Write the
standard basis of $R^{2{s+1\choose 2}}$ as 
\[
     B_2 = \left\{    \begin{array}{ll} 
	  e_{1,(\ell_2+1,\ell_3),\ell_1-1 }\\
	  e_{2,(\ell_1,\ell_2),\ell_3 }
	   \end{array}
	 : \; 1 \leq \ell_1  \leq s,\; 0 \leq \ell_2,\ell_3 \leq s-1 \right\} 
\]
and define $\partial_2 : R^{2{s+1\choose 2}} \to R^{s+2 \choose
2}$ by 
\begin{eqnarray*}
\partial_2(e_{1,(\ell_2+1,\ell_3),\ell_1-1 }) & = & x_2  \cdot e_{\ell_1-1,\ell_2+1,\ell_3} - x_1  \cdot e_{\ell_1-1,\ell_2,\ell_3+1},\\
\partial_2(e_{2,(\ell_1,\ell_2),\ell_3 }) & = & x_3 \cdot e_{\ell_1,\ell_2,\ell_3} - x_2  \cdot e_{\ell_1-1,\ell_2+1,\ell_3}.
\end{eqnarray*}
From the equation,
$\beta_3 - {2(s+1) \choose 2} + {s+2 \choose 2} -1 = 0$, it follows
that $\beta_3 = {s \choose 2}$. Now, we describe $\ker \partial_2$.
For $ 2 \leq \ell_1  \leq s \text{ and } 0 \leq \ell_2, \ell_3 \leq
s-2$, set
\[
 E_{\ell_1,\ell_2,\ell_3}= -x_3 \cdot e_{2,(\ell_1-1,\ell_2),\ell_3 +1 } + x_2 \cdot 
 e_{2,(\ell_1-1,\ell_2+1),\ell_3 } +x_2 \cdot e_{1,(\ell_2+2,\ell_3),\ell_1-2 } -x_3 \cdot e_{1,(\ell_2+1,\ell_3),\ell_1-1 }. 
 \] Note that 
 \[
  B_3 = \{ E_{\ell_1,\ell_2,\ell_3} ~ : ~ 2 \leq \ell_1  \leq s,\; 0 \leq \ell_2,\ell_3 \leq s-2\} \subset \ker \partial_2.
 \]
Let $N = \ker \partial_2$. We claim that $\bar{B}_3 =
\{\bar{E}_{\ell_1,\ell_2,\ell_3}\}$ is a basis for $N/\m N$. It is
enough to prove that $\bar{B}_3$ is $R/\m$-linearly independent.
Suppose $\displaystyle{\sum_{\ell_1,\ell_2,\ell_3}
\alpha_{\ell_1,\ell_2,\ell_3}\bar{E}_{\ell_1,\ell_2,\ell_3} =
\bar{0}}$. Since $B_3 \subset N_1$ and $\m N \subset \oplus_{r\geq
2}N_r$, the above equation implies that $\sum\alpha_{\ell_1,\ell_2,\ell_3}
E_{\ell_1,\ell_2,\ell_3} = 0$.
In the above equation, the coefficient of
$e_{2,(i,j),k} = -x_3 \alpha_{i+1,j,i-1} + x_2
\alpha_{i+1,j-1,k}$ and coefficient of $e_{1,(i,j),k} = x_2
\alpha_{k+2,j-2,i} - x_3 \alpha_{k+1,j-1,i}$. Since the set $B_2$ is
linearly independent in $R^{2{s+1 \choose 2}}$, all these coefficients
have to be zero, i.e., $\alpha_{i,j,k} = 0$ for all $i, j, k$. This
proves our claim and hence $N = \langle B_3 \rangle$.
%
%
%Since $e_{2,(1,0),s-2}$ occurs only in the expression for
%$E_{2,0,s-2}$, it can be seen that $\alpha_{2,0,s-2} = 0$.
%
%
%
%Since $|B_3| = {s \choose 2} = \mu(\ker \partial_2)$, $\ker \partial_2 = \langle B_3 \rangle$. 
Let
$\{ H_{\ell_1,\ell_2,\ell_3}  ~ : ~ 2 \leq \ell_1  \leq s,\; 0 \leq \ell_2,\ell_3 \leq s-2\}$ denote the standard basis
of $R^{s \choose 2}$. Define $\partial_3 : R^{s \choose 2} \to
R^{2(s+1) \choose 2}$ by $ \partial_3(H_{\ell_1,\ell_2,\ell_3}) =
E_{\ell_1,\ell_2,\ell_3}.$  Since $\pdim(R/I^s) = 3$, this map is
injective and the resolution is complete.

It can also be seen that $\deg e_{\ell_1,\ell_2,\ell_3} = 2s, ~ \deg e_{1,(\ell_2+1,\ell_3),\ell_1-1 } = \deg
e_{2,(\ell_1,\ell_2),\ell_3 } = 2s+1$ and $\deg H_{\ell_1,\ell_2,\ell_3} = 2s+2$. Therefore
we get the minimal graded free resolution of $R/I^s$ as
\[
  0 \to R(-2s-2)^{s \choose 2} \overset{\partial_3}{\longrightarrow}
  R(-2s-1)^{2(s+1) \choose 2}
  \overset{\partial_2}{\longrightarrow} R(-2s)^{s+2 \choose 2}
  \overset{\partial_1}{\longrightarrow} R \to 0.
\]
Therefore, the Hilbert series of $R/I^s$ is given by
\begin{eqnarray*}
  H(R/I^s, t) & = & \frac{1-{s+2 \choose 2}t^{2s} + 2{s+1 \choose
  2}t^{2s+1} - {s \choose 2}t^{2s+2}}{(1-t)^3}\\
  & = & (1-t)^{-2}\frac{(1-{s+2 \choose 2}t^{2s} + 2{s+1 \choose
  2}t^{2s+1} - {s \choose 2}t^{2s+2})}{(1-t)}.
\end{eqnarray*}
By expanding $(1-t)^{-2}$ in the power series form and multiplying
with the numerator, we get the required expression.
\end{proof}
We now proceed to compute the minimal graded free resolution of
powers of complete tripartite graphs.

\vskip 2mm
\begin{notation}\label{3-partite-notn}
Let $G$ denote a complete tripartite graph with $V(G) = V_1 \sqcup
V_2 \sqcup V_3$ and $E(G) = \{\{a,b\} ~ : ~ b \in V_i, b \in V_j, i
\neq j \}$. Set $V_1 = \{x_1, \ldots, x_\ell\}, ~ V_2 = \{y_1, \ldots,
y_m \}$ and $V_3 = \{z_1, \ldots, z_n\}$. Let $J_G$ denote the vertex
cover ideal of $G$. Let $X = \prod_{i=1}^\ell x_i, ~ Y = \prod_{j=1}^m
y_i$ and $Z = \prod_{k=1}^n z_i$. It can be seen that $J_G =(XY, XZ,
YZ)$.
\end{notation}
\begin{theorem}\label{3-partite} Let $R = K[x_1, \ldots x_{m_1}, y_1,\ldots,y_{m_2}, z_1,
\ldots, z_{m_3}]$. Let $G$ be a complete tripartite graph as in Notation
\ref{3-partite-notn}. Let $J_G \subset R$ denote the cover ideal of
$G$. Then for all $s \geq 2$, the minimal free resolution of
$R/J_G^s$ is of the form:
\[
  0 \to R^{s \choose 2} \overset{\partial_3}{\longrightarrow}
  R^{2(s+1) \choose 2}
  \overset{\partial_2}{\longrightarrow} R^{s+2 \choose 2}
  \overset{\partial_1}{\longrightarrow} R \to 0.
\]
Set $\alpha = (s-\ell_3)m_1 +(s-\ell_2)m_2+ (s-\ell_1)m_3$. 
Then for all $s \geq 2$,
\[
\reg(J_G^s) = \max \left\{
  \begin{array}{ll} 
	\alpha, & \text{  for  } 0 \leq \ell_1, \ell_2,\ell_3 \leq s, \\
	\alpha + m_3 -1, & \text{  for  } 1 \leq \ell_1  \leq s,\; 0 \leq \ell_2,\ell_3 \leq s-1,\\
	\alpha + 2m_3 -2, & \text{  for } 2 \leq \ell_1  \leq s,\; 0 \leq \ell_2,\ell_3 \leq s-2. 
  \end{array}
\right.
\]
\end{theorem}
\begin{proof}
  Taking $X_1 = x_1, ~X_2 = x_2$ and $X_3 = x_3$ in Theorem \ref{C3}, it
follows that the minimal free resolution of $S/J_G^s$ is of the given
form: 
\[
  0 \to R^{s \choose 2} \overset{\partial_3}{\longrightarrow}
  R^{2(s+1) \choose 2}
  \overset{\partial_2}{\longrightarrow} R^{s+2 \choose 2}
  \overset{\partial_1}{\longrightarrow} R \to 0.
\]
We now compute the degrees of the generators and hence obtain
the regularity. Let $ \deg X_1 = m_1, \deg X_2=m_2,$ and $\deg X_3=m_3$. Then it follows that 
\[
 \deg e_{\ell_1,\ell_2,\ell_3}  =  
 \deg \left( X_1^{s-\ell_3}X_2^{s- \ell_3}X_3^{s- \ell_1} \right) = (s-\ell_3)m_1 +(s-\ell_2)m_2+ (s-\ell_1)m_3.
\]
 We observe that 
\begin{eqnarray*}
 & \deg e_{1,(\ell_2+1,\ell_3),\ell_1-1 } =  \deg e_{2,(\ell_1,\ell_2),\ell_3 }  = 
 \deg e_{\ell_1,\ell_2,\ell_3} + \deg X_3, \\
 & \deg H_{\ell_1,\ell_2,\ell_3}  = \deg e_{\ell_1,\ell_2,\ell_3} + 2 \deg X_3.
\end{eqnarray*}
Therefore, 
\[
\reg(J_G^s) = \max \left\{
  \begin{array}{ll} 
	\deg e_{\ell_1,\ell_2,\ell_3},     & \text{ for }  0 \leq \ell_1, \ell_2, \ell_3 \leq s, \\
 \deg e_{\ell_1,\ell_2,\ell_3} + \deg(X_3)-1,  & \text{ for } 1 \leq \ell_1 \leq s, \; 
 0 \leq \ell_2, \ell_3  \leq s-1, \\
	\deg e_{\ell_1,\ell_2,\ell_3} + 2\deg(X_3)-2, & \text{ for }  2 \leq \ell_1 \leq s, 
	\; 0 \leq \ell_2, \ell_3  \leq s-2, 
  \end{array}
\right.
\]
where $ \ell_1 + \ell_2 +\ell_3 =s$. Let $\alpha = \deg e_{\ell_1,\ell_2,\ell_3}$. Then the regularity assertion follows. 
\end{proof}

Observe that the above expression for $\reg(J_G^s)$ is not really in
the form $ds + e$. Given a graph, one can compute $d$ and $e$ by
studying the interplay between the cardinality of the partitions.
For example, suppose
the graph is unmixed, i.e., all the partitions are of same
cardinality. Then, $\reg(J_G^s) = 2\ell s + (2\ell-2)$ for all $s \geq
2$, where $\ell = m_1 = m_2 = m_3$. Note also that the stabilization
index in this case is $2$. As done in Discussion (\ref{disc}),
one can derive various expressions for $\reg(J_G^s)$ for different cases as well.
Consider the arithmetic progression 
$m_1 = m+2r$, $m_2 =m$, and $ m_3 =m+r$:

\begin{corollary} Let $m,r$ be any two positive integers. Let
$m_1 = m+2r$, $m_2 =m$, and $ m_3 =m+r$ in Theorem \ref{3-partite}. 
Then for all $s \geq 2$, we have
\[
 \reg(J_G^s) = s(2m+3r)+2m-2.
\]
\end{corollary}
\begin{proof} By Theorem \ref{3-partite}, $ \alpha = (s-\ell_3)m_1 +(s-\ell_2)m_2+ (s-\ell_1)m_3$. Hence we get 
\[
  \alpha = s(3m+3r) -\ell_3 (m+2r) -\ell_2 (m) -\ell_1 (m+r).
\]
Using Theorem \ref{3-partite}, we have for all $s \geq 2$,
\[
\reg(J_G^s) = \max \left\{
  \begin{array}{ll} 
	\alpha,               & \text{ for }  0 \leq \ell_1, \ell_2, \ell_3  \leq s, \\
	\alpha + (m+r)-1,  & \text{ for }  1 \leq \ell_1 \leq s, \; 0 \leq \ell_2, \ell_3 \leq s-1, \\
	\alpha + 2(m+r)-2, & \text{ for }  2 \leq \ell_1 \leq s, \; 0 \leq \ell_2, \ell_3 \leq s-2, 
	 \end{array}
\right.
\]
where $ \ell_1 + \ell_2 +\ell_3=s$. Since regularity $\reg(J_G^s)$ is maximum of all the numbers, 
we need to maximize the value of $\alpha$. For this to happen, negative terms in $\alpha$ should 
be minimum. The coefficient of $\ell_3$ is largest among the negative terms in $\alpha$, 
so $\ell_3$ should be assigned the least value. After $\ell_3$, assign the minimum value to 
$\ell_1$, and finally take $ \ell_2 = s-\ell_1 -\ell_3 $. For example, to get the maximum 
of $\alpha$ when $2 \leq \ell_1 \leq s, \; 0 \leq \ell_2, \ell_3 \leq s-2$, put $\ell_3=0,\ell_1=2$, and 
$\ell_2=s-2$. We get for all $s \geq 2$
\[
\reg(J_G^s) = \max \left\{
  \begin{array}{ll} 
	s (2m+3r),               & \text{ for }  0 \leq \ell_1, \ell_2, \ell_3, \ell_4 \leq s, \\
	s (2m+3r)+m-1,  & \text{ for }  1 \leq \ell_1 \leq s, \; 0 \leq \ell_2, \ell_3, \ell_4 \leq s-1, \\
	s (2m+3r)+2m-2, & \text{ for }  2 \leq \ell_1 \leq s, \; 0 \leq \ell_2, \ell_3, \ell_4 \leq s-2. 
	 \end{array}
\right.
\]
For all $m \geq 1$, and for all $s \geq 2$, we get 
\[
 \reg(J_G^s) = s (2m+3r)+2m-2.
\]
\end{proof}
\subsection{Complete 4-partite graphs}
We now describe the resolution and regularity of powers of cover
ideals of $4$-partite graphs. For this purpose, we first study the
resolution of powers of cover ideal of the complete graph $K_4$.
%Let $R = K[x_1, x_2, x_3, x_4]$ and $I = (x_1x_2x_3, x_1x_2x_4,
%x_1x_3x_4, x_2x_3x_4)$. Then $I$ is the edge ideal of a complete
%$3$-uniform hypergraph $K_4^{(3)}$. In this section we explicitly 
%describe the minimal free resolution of $R/I^s$ for all $s \geq 3$. 

\begin{theorem}\label{C4} 
Let $R = K[x_1, x_2, x_3, x_4]$ and $I = (x_1x_2x_3, x_1x_2x_4,
x_1x_3x_4, x_2x_3x_4)$. Then, for $s \geq 3$, the minimal graded free
resolution of $R/I^s$ is of the form 
\[
  0 \to R(-3s-3)^{\beta_4} {\longrightarrow}   R(-3s-2)^{\beta_3} {\longrightarrow}
  R(-3s-1)^{\beta_2}
 {\longrightarrow} R(-3s)^{\beta_1}
  {\longrightarrow} R \to 0
\]
where 
\[
 \beta_1 = { s+3 \choose 3}, ~ \beta_2 = 3{ s+2 \choose 3},~\beta_3 = 3{ s+1 \choose 3}, \text{ and } \beta_4 = { s \choose 3}.
\]
In particular, $\reg_R (R/I^s)=3s$. Moreover the Hilbert series of $R/I^s$ is given by 
\begin{eqnarray*}
 H(R/I^s,t) & = & \frac{1+2t+3t^2+4t^3+ \cdots +3st^{3s-1} - \left({s +3 \choose 3}-3s-1  \right)t^{3s} + {s \choose 3} t^{3s+1}}{(1-t)^2} \\
& = & \frac{ \sum_{i=0}^{3s-1} (i+1)t^i - (3s+1 - {s +3 \choose 3}  )t^{3s} + {s \choose 3} t^{3s+1}}{(1-t)^2}.
\end{eqnarray*}
\end{theorem}
\begin{proof} By Remark \ref{rem:lin res powers}, $I^s$ has linear
resolution for all $s \geq 1$. Moreover, by Remark \ref{rem:depth
powers}, $\depth R/I^s = 0$ for all $s \geq 3$ and hence $\pdim R/I^s
= 4$ for all $s \geq 3$.  Therefore, the minimal graded free
resolution of $R/I^s$, for $s \geq 3$, is of the form 
\[
  0 \to R(-3s-3)^{\beta_4} \overset{\partial_4}{\longrightarrow}  R(-3s-2)^{\beta_3} \overset{\partial_3}{\longrightarrow}
  R(-3s-1)^{\beta_2}
  \overset{\partial_2}{\longrightarrow} R(-3s)^{\beta_1}
  \overset{\partial_1}{\longrightarrow} R \to 0.
\]
Let $g_1 = x_1x_2x_3, g_2 = x_1x_2x_4, g_3=x_1x_3x_4$ and $g_4 =
x_2x_3x_4$. The minimal generators of $I^s$ are of the form
$g_1^{\ell_1}g_2^{\ell_2}g_3^{\ell_3}g_4^{\ell_4}$ where $0 \leq
\ell_i \leq s$ for every $i$ and $\ell_1 + \ell_2 +\ell_3 + \ell_4=s$.
The number of non-negative integral solution to the 
linear equation $\ell_1 + \ell_2 +\ell_3 + \ell_4=s$ is ${s+3 \choose 3}$. Hence we have 
$\mu(J^s)= \beta_1 = {s+3 \choose 3}$. Note that
$g_1^{\ell_1}g_2^{\ell_2}g_3^{\ell_3}g_4^{\ell_4}=x_1^{s-\ell_4}x_2^{s-
\ell_3}x_3^{s- \ell_2}x_4^{s- \ell_1} $ with $0 \leq
\ell_1,\ell_2,\ell_3,\ell_4 \leq s$. Set
$f_{\ell_1,\ell_2,\ell_3,\ell_4}=x_1^{s-\ell_4}x_2^{s- \ell_3}x_3^{s-
\ell_2}x_4^{s- \ell_1} $. Let \[
  \{e_{\ell_1,\ell_2,\ell_3,\ell_4} ~ : ~ 0 \leq
  \ell_1,\ell_2,\ell_3,\ell_4 \leq s \text{ and } \ell_1 + \ell_2
+\ell_3 + \ell_4=s \}\]
denote the standard basis of $R^{s+3 \choose 3}$ and consider
the map $\partial_1 : R^{s+3 \choose 3} \to R$ defined by $\partial_1(e_{\ell_1,\ell_2,\ell_3,\ell_4}) = f_{\ell_1,\ell_2,\ell_3,\ell_4}$. Now we find the minimal generators for $\ker \partial_1$.

Let $f_{\ell_1,\ell_2,\ell_3,\ell_4}, f_{t_1,t_2,t_3,t_4}$ be any two
minimal
monomial generators of $I^s$. It is known that the kernel of $\partial_1$ is
generated by binomials of the form
$m_{\ell_1,\ell_2,\ell_3,\ell_4}f_{\ell_1,\ell_2,\ell_3,\ell_4} -
m_{t_1,t_2,t_3,t_4}f_{t_1,t_2,t_3,t_4}$, where
$m_{i,j,k,l}$'s are
monomials in $R$. Since the syzygy is generated by linear binomials, 
we need to find
conditions on $\ell_1,\ell_2,\ell_3, \ell_4$ such that
$\frac{f_{\ell_1,\ell_2,\ell_3,\ell_4}}{f_{t_1,t_2,t_3,t_4}}$ is equal
to $\frac{x_i}{x_j}$ for some $i,j$. Observe that
\[
 \frac{f_{\ell_1,\ell_2,\ell_3,\ell_4}}{f_{t_1,t_2,t_3,t_4}} = x_1^{t_4-\ell_4} x_2^{t_3-\ell_3} x_3^{t_2-\ell_2} x_4^{t_1-\ell_1}.
\]
Suppose $t_4=\ell_4 +1$, $t_{3}=\ell_3-1$ and $t_j=\ell_j$ for $j=1,2$. Then we get the linear syzygy relation
\begin{eqnarray}\label{eqn:beta_2 relations00}
 x_{2} f_{\ell_1,\ell_2,\ell_3,\ell_4} - x_1 f_{\ell_1,\ell_{2}, \ell_3-1, \ell_4 +1}=0.
\end{eqnarray}
Fixing $t_3$ and $t_4$, with $1 \leq t_3 + t_4 \leq s$, there are as
many linear syzygies are there as the number of solutions of $t_1+t_2
= s-(t_3+t_4)$. Therefore, for the pair $(t_3,t_4)$, there are ${s+1
\choose 2} + {s \choose 2} + \cdots + {2 \choose 2} = {s+2 \choose 3}$
number of solutions. Similarly for each pair $(t_4, t_2), (t_4, t_1),
(t_3, t_2), (t_3, t_1)$ and $(t_2,t_1)$, we get ${s+2 \choose 3}$
linear syzygies. 
Note that the syzygies $x_4f_{t_1+1,t_2-1,t_3,t_4} -
x_3f_{t_1,t_2,t_3,t_4}$ and $x_3f_{t_1,t_2,t_3,t_4} -
x_2f_{t_1,t_2-1,t_3+1,t_4}$ give rise to another linear syzygy
$x_4f_{t_1+1,t_2-1,t_3,t_4} -x_2f_{t_1,t_2-1,t_3+1,t_4}$. The same
linear syzygy can also be obtained from a combination of linear
syzygies that arise out of the pairs $(t_1,t_4)$
and $(t_3,t_4)$. Therefore, to get a minimal
generating set, we only need to consider linear syzygies corresponding
to the pairs $(t_1,t_2), ~(t_2,t_3)$ and $(t_3,t_4)$. For each such
pair, we have ${s+2 \choose 3}$ number of linear syzygies. Hence
$\beta_2 = 3 {s+2 \choose 3}$.

Write the basis elements of $R^{3{s+2\choose 3}}$ as 
\[
     B_2 = \left\{    \begin{array}{ll} 
	  e_{(1,\ell_1-1,\ell_2),\ell_3+1,\ell_4}\\
	  e_{(2,\ell_1-1,\ell_4),\ell_2+1,\ell_3}\\
	  e_{(3,\ell_3,\ell_4),\ell_1,\ell_2}
	  \end{array}
  : \; 1 \leq \ell_1  \leq s,\; 0 \leq \ell_2,\ell_3,\ell_4 \leq s-1
\right\} 
\]
and define
$\partial_2 : R^{3{s+1\choose 2}} \longrightarrow R^{\beta_1}$ by
\begin{eqnarray*}
\partial_2(e_{(1,\ell_1-1,\ell_2),\ell_3+1,\ell_4}) & = & x_{2}
e_{\ell_1-1,\ell_2,\ell_3+1,\ell_4} - x_1 e_{\ell_1-1,\ell_{2},
\ell_3, \ell_4 +1}; \\
\partial_2(e_{(2,\ell_1-1,\ell_4),\ell_2+1,\ell_3}) & = & x_{3}
e_{\ell_1-1,\ell_2+1,\ell_3,\ell_4} - x_2 e_{\ell_1-1,\ell_{2},
\ell_3+1, \ell_4};\\
\partial_2(e_{(3,\ell_3,\ell_4),\ell_1,\ell_2}) & = & x_{4} e_{\ell_1,\ell_2,\ell_3,\ell_4} - x_3 e_{\ell_1-1,\ell_{2}+1, \ell_3, \ell_4}.
\end{eqnarray*}
Now we decipher the Betti numbers $\beta_3$ and $\beta_4$ to complete the resolution. The Hilbert series of $R/I^s$ is 
\[
 H(R/I^s,t) = \frac{{ \left( 1- \beta_1 t^{3s} + \beta_2 t^{3s+1} - \beta_3t^{3s+2} +\beta_4t^{3s+3} \right)}}{(1-t)^4}=\frac{p(t)}{(1-t)^4}.
\]
Since $\dim R/I^s=2$, the polynomial $p(t)$ has a factor $(1-t)^2$. Note that $(1-t)^2$ is a monic polynomial of degree $2$, hence we can write $p(t)= (1-t)^2 \cdot q(t)$ where 
\begin{eqnarray}\label{eqn:q(t) polynomial}
q(t)=\beta_4 t^{3s+1}+ (2\beta_4-\beta_3)t^{3s}+a_{3s-1}t^{3s-1}+\cdots + a_1 t + 1.
\end{eqnarray}
On the other hand, we also have 
\begin{eqnarray}\label{eqn:q(t) polynomial from HS}
 \frac{p(t)}{(1-t)^{2}} & = &
 \left(\sum_{n\geq 0}(n+1)t^n\right) \left( 1- \beta_1t^{3s} +
 \beta_2t^{3s+1} - \beta_3t^{3s+2} +\beta_4t^{3s+3} \right)\\
 & =& \sum_{n=0}^{3s-1}(n+1)t^n+ (3s+1 - \beta_1 ) t^{3s}+ (3s+2
 -2\beta_1 +\beta_2)t^{3s+1} + \sum_{j\geq 3s+2}a_{j}t^{j} \nonumber
\end{eqnarray}
For the expressions in Equations $(\ref{eqn:q(t) polynomial})$ and
$(\ref{eqn:q(t) polynomial from HS})$ to be equal, their respective
coefficients should be equal. In particular, we should have that 
\[
 \beta_4= 3s+2 -2\beta_1 +\beta_2 \text{ and } 2\beta_4-\beta_3 = 3s+1 - \beta_1.
\]
On substituting $\beta_1 = { s+3 \choose 3}$ and $\beta_2 = 3 {s+2 \choose 3}$, we get 
$\beta_4= {s \choose 3}$ from the first equation and substituting the
value in the second equation, we get $\beta_3=3{s+1 \choose 3} $. 
Now we verify that $a_j=0$ for all $j \geq 3s+2$. Note that for $r
\geq 2$, the coefficient of $t^{3s+r}$ in Equation (\ref{eqn:q(t) polynomial from HS}) is
\[
 a_{3s+r}= (3s+r+1) + (r-2) \beta_4  -(r-1) \beta_3  + r \beta_2 - ( r+1) \beta_1.
\]
Since, $1 - \beta_1 + \beta_2 - \beta_3 + \beta_4 = 0$, this equation
is
reduced to $a_{3s+r}= (3s+3) - \beta_3  + 2 \beta_2 - 3 \beta_1.$
Applying the binomial identify $ {n+1 \choose r+1} = {n \choose
r+1}+{n \choose r}$ repeatedly, we get $a_{3s+r}=0$ for all $r \geq 2$. Hence 
the Hilbert series of $R/I^s$ is  
\[
 H(R/I^s,t) = \frac{1+2t+3t^2+4t^3+ \cdots +3st^{3s-1} - \left({s +3 \choose 3}-3s-1  \right)t^{3s} + {s \choose 3} t^{3s+1}}{(1-t)^2}.
\]

We now complete the description of the resolution.
Write the basis elements of 
$R^{3{s+1 \choose 3}}$ as 
\[
   B_3 = \left\{   \begin{array}{ll} 
	  E_{1,\ell_1,\ell_2,\ell_3,\ell_4};\\
	  E_{2,\ell_1,\ell_2,\ell_3,\ell_4},\\
	  E_{3,\ell_1,\ell_2,\ell_3,\ell_4},
	  \end{array}
  : \; 2 \leq \ell_1 \leq s, \; 0 \leq \ell_2, \ell_3, \ell_4 \leq s-2 \right\}.
\]
Note that $|B_3|=3{s+1 \choose 3}$. Now define the map $\partial_3 : R^{3{s+1 \choose 3}} \longrightarrow R^{3 {s+2 \choose 3}}$ by 
\begin{eqnarray*}
\partial_3(E_{1,\ell_1,\ell_2,\ell_3,\ell_4}) & = & x_{4} e_{(2,\ell_1-1,\ell_4),\ell_2+1,\ell_3 } + x_4 
e_{(1,\ell_1-1,\ell_2),\ell_3+1,\ell_4}-x_3 e_{(3,\ell_3,\ell_4),\ell_1-1,\ell_2+1} \\ & &
-x_3 e_{(2,\ell_1-2,\ell_4),\ell_2+2,\ell_3 } -x_3 e_{(1,\ell_1-2,\ell_2+1),\ell_3+1,\ell_4} +x_1 e_{(3,\ell_3,\ell_4 +1),\ell_1-1,\ell_2};\\
\partial_3(E_{2,\ell_1,\ell_2,\ell_3,\ell_4}) & = & x_{4} e_{(1,\ell_1-1,\ell_2),\ell_3+1,\ell_4} - x_3 e_{(1,\ell_1-2,\ell_2+1),\ell_3+1,\ell_4} - x_2 e_{(3,\ell_3+1,\ell_4),\ell_1-1,\ell_2} \\ & &
+x_1 e_{(3,\ell_3,\ell_4+1),\ell_1-1,\ell_2};\\
\partial_3(E_{3,\ell_1,\ell_2,\ell_3,\ell_4}) & = & x_{3} e_{(1,\ell_1-2,\ell_2+1),\ell_3+1,\ell_4} - x_2 e_{(2,\ell_1-2,\ell_4),\ell_2+1,\ell_3+1 } -x_2 e_{(1,\ell_1-2,\ell_2),\ell_3+2,\ell_4} \\ & &  + x_1 e_{(2,\ell_1-2,\ell_4+1),\ell_2+1,\ell_3 }.
\end{eqnarray*}
We now compute the kernel of $\partial_3$. Consider the set
\iffalse
We know that $\beta_4= {s \choose 3}$. Since $\pdim R/I^s = 4$, the image of the map $\partial_4 : R^{{s \choose 3}} \longrightarrow R^{3 {s+1 \choose 3}}$ is the kernel of the map $\partial_3 : R^{3{s+1 \choose 3}} \longrightarrow R^{3 {s+2 \choose 3}}$. Hence to describe the map $\partial_4$, it is enough to compute the kernel of the map $\partial_3$. The kernel of the map $\partial_3$ is given by 
\fi
\[
   A = \left\{   
	   H_{\ell_1,\ell_2,\ell_3,\ell_4}
  :  3 \leq \ell_1 \leq s, \; 0 \leq \ell_2, \ell_3, \ell_4 \leq s-3 \right\},
\]
where 
\begin{eqnarray*}
  H_{\ell_1,\ell_2,\ell_3,\ell_4} & = & x_4 E_{3,\ell_1,\ell_2,\ell_3,\ell_4} -x_3 
  E_{2,\ell_1-1,\ell_2+1,\ell_3,\ell_4}  - x_3 E_{3,\ell_1-1,\ell_2+1,\ell_3,\ell_4} 
	  + x_2 E_{1,\ell_1-1,\ell_2,\ell_3+1,\ell_4} \\ & & - x_1 E_{1,\ell_1-1,\ell_2,\ell_3,\ell_4+1}  + 
 x_1 E_{2,\ell_1-1,\ell_2,\ell_3,\ell_4+1}.
\end{eqnarray*}
%where $ 3 \leq \ell_1 \leq  s$, and $0 \leq \ell_2,\ell_3,\ell_4 \leq s-3$. 
It can be verified that $\partial_3
(H_{\ell_1,\ell_2,\ell_3,\ell_4})=0$, i.e.,  $A \subseteq \ker \partial_3$. 
%To say that the 
%set $A$ is precisely the kernel of the map $\partial_3$, one needs to show that $\mu(A)={s \choose 3}$. 
Let $\ell_1^{\prime}=\ell_1-3$, then one has $ \ell_1^{\prime} +
\ell_2 +\ell_3 +\ell_4 =s-3$. The cardinality of $A$ is equal to the
total number of non-negative integral solution of this linear equation
which is ${s \choose 3}$. As in the proof of Theorem \ref{C3}, it can
be seen that 
$\ker \partial_3 = \langle A\rangle$. Write the basis elements of $R^{{s \choose 3}}$ as $B_4= \left\{ G_{\ell_1,\ell_2,\ell_3,\ell_4}  : \; 3 \leq \ell_1 \leq s, \; 0 \leq \ell_2, \ell_3, \ell_4 \leq s-3 \right\} $
and define the map $\partial_4 : R^{{s \choose 3}} \longrightarrow R^{3 {s+1 \choose 3}}$ by 
\begin{eqnarray*}
\partial_4(G_{\ell_1,\ell_2,\ell_3,\ell_4}) = &  H_{\ell_1,\ell_2,\ell_3,\ell_4}.
\end{eqnarray*}
This is an injective map and hence we get the complete resolution:
\[
   0 \to R(-3s-3)^{\beta_4} {\longrightarrow}   R(-3s-2)^{\beta_3} {\longrightarrow}
  R(-3s-1)^{\beta_2}
 {\longrightarrow} R(-3s)^{\beta_1}
  {\longrightarrow} R \to 0.
\]
\end{proof} 

Note that in the above proof, we used $s \geq 3$ only to conclude that
$\pdim(R/I^s) = 4$. By Remark \ref{rem:depth powers}, $\depth(R/I) =
2$ and hence $\pdim(R/I) = 2$. Similarly, $\depth(R/I^2) = 1$ and
hence $\pdim(R/I^2) = 3$. This forces $\partial_2$ to be injective
when $s = 1$ and $\partial_3$ to be injective when $s = 2$.
The computations of syzygies in the cases of
resolution of $R/I$ and $R/I^2$ remain the same as given in the above
proof. Therefore, we get resolutions truncated at
$R^{\beta_2}$ in the case of $R/I$ and truncated at $R^{\beta_3}$ in
the case of $R/I^2$, with the expressions for $\beta_2$ and $\beta_3$
coinciding with the ones given in the proof. Therefore, we can
conclude that in this case, $\reg(I^s) = 3s$ for all $s \geq 1$.

As an immediate consequence, we obtain an expression for the
asymptotic regularity of cover ideals of complete $4$-partite graphs.

\begin{theorem}\label{4-partite}
Let $G$ denote a complete $4$-partite graph with $V(G) = \sqcup_{i=1}^4
V_i$ and $E(G) = \{\{a,b\} ~ : ~ a \in V_i, b \in V_j, i \neq j \}$. 
Set $V_i = \{x_{i1}, \ldots, x_{im_i}\}$ for $i = 1, \ldots, 4$. Let
$J_G \subset R = K[x_{ij} : 1 \leq i \leq 4; ~1\leq j \leq m_i]$ 
denote the cover ideal of $G$. Then the minimal free resolution of $R/J_G^s$ is of the form:
\[
0 \to R^{s \choose 3} {\longrightarrow}   R^{3{s+1\choose 3}} {\longrightarrow}
  R^{3{s+2 \choose 3}}
 {\longrightarrow} R^{s+3 \choose 3}
  {\longrightarrow} R \to 0.
\]
Set $ \alpha = (s-\ell_4)m_1 +(s-\ell_3)m_2+ (s-\ell_2)m_3 +
(s-\ell_1)m_4$. Furthermore, we have 
\[
\reg(J_G^s) = \max \left\{
  \begin{array}{ll} 
	 \alpha,               & \text{ for }  0 \leq \ell_1, \ell_2, \ell_3, \ell_4 \leq s, \\
	\alpha + m_4-1,  & \text{ for }  1 \leq \ell_1 \leq s, \; 0 \leq \ell_2, \ell_3, \ell_4 \leq s-1, \\
	\alpha + 2m_4-2, & \text{ for }  2 \leq \ell_1 \leq s, \; 0 \leq \ell_2, \ell_3, \ell_4 \leq s-2, \\
	\alpha + 3m_4-3, & \text{ for }  3 \leq \ell_1 \leq s, \; 0 \leq \ell_2, \ell_3, \ell_4 \leq s-3, 
  \end{array}
\right.
\]
where $ \ell_1 + \ell_2 +\ell_3 +\ell_4 =s$.
\end{theorem}
\begin{proof}Let $X_i = \prod_{j=1}^{m_i} 
x_{ij}$. Then $J_G = (X_1X_2X_3, X_1X_2X_4, X_1X_3X_4, X_2X_3X_4)$.
Then it
follows from Theorem \ref{C4} that the minimal free resolution of $R/J_G^s$ is of the given
form
\[
0 \to R^{s \choose 3} {\longrightarrow}   R^{3{s+1\choose 3}} {\longrightarrow}
  R^{3{s+2 \choose 3}}
 {\longrightarrow} R^{s+3 \choose 3}
  {\longrightarrow} R \to 0.
\]
To compute the regularity of $R/J_G^s$, we first need to find the degree's of the generators 
of the syzygies. 
Following the notation of Theorem \ref{C4}, we have 
\begin{eqnarray*}
 \deg e_{\ell_1,\ell_2,\ell_3,\ell_4}  & = &  \deg \left( X_1^{s-\ell_4}X_2^{s- \ell_3}X_3^{s- \ell_2}X_4^{s- \ell_1} 
 \right) = (s-\ell_4)m_1 +(s-\ell_3)m_2+ (s-\ell_2)m_3 +
 (s-\ell_1)m_4,
 \\
 \deg  e_{(1,\ell_1-1,\ell_2),\ell_3,\ell_4} & =  & \deg e_{(2,\ell_1-1,\ell_4),\ell_2+1,\ell_3}  = 
 \deg e_{(3,\ell_3,\ell_4),\ell_1,\ell_2} =  \deg
 e_{\ell_1,\ell_2,\ell_3,\ell_4} + \deg(X_4),\\
 \deg E_{1,\ell_1,\ell_2,\ell_3,\ell_4} & = & \deg
 E_{2,\ell_1,\ell_2,\ell_3,\ell_4}  = \deg
 E_{3,\ell_1,\ell_2,\ell_3,\ell_4} = \deg
 e_{\ell_1,\ell_2,\ell_3,\ell_4} + 2\deg(X_4), \\
 \deg G_{\ell_1,\ell_2,\ell_3,\ell_4} & = &  \deg e_{\ell_1,\ell_2,\ell_3,\ell_4} + 3\deg(X_4).
\end{eqnarray*}
%Another interesting observation is the following:
%\begin{eqnarray*}
% \deg  e_{(1,\ell_1-1,\ell_2),\ell_3,\ell_4} & = \deg e_{\ell_1,\ell_2,\ell_3,\ell_4} + \deg(X_4),\\
% \deg E_{1,\ell_1,\ell_2,\ell_3,\ell_4} & = \deg e_{\ell_1,\ell_2,\ell_3,\ell_4} + 2\deg(X_4),\\
% \deg G_{\ell_1,\ell_2,\ell_3,\ell_4} & =  \deg e_{\ell_1,\ell_2,\ell_3,\ell_4} + 3\deg(X_4).
%\end{eqnarray*}
Therefore, by setting $\alpha = \deg e_{\ell_1,
\ell_2,\ell_3,\ell_4}$, we get
\[
\reg(J_G^s) = \max \left\{
  \begin{array}{ll} 
	\alpha,     & \text{ for }  0 \leq \ell_1, \ell_2, \ell_3, \ell_4 \leq s, \\
	\alpha + \deg(X_4)-1,  & \text{ for }  1 \leq \ell_1 \leq s, \; 0 \leq \ell_2, \ell_3, \ell_4 \leq s-1, \\
	\alpha + 2\deg(X_4)-2, & \text{ for }  2 \leq \ell_1 \leq s, \; 0 \leq \ell_2, \ell_3, \ell_4 \leq s-2, \\
	\alpha + 3\deg(X_4)-3, & \text{ for }  3 \leq \ell_1 \leq s, \; 0 \leq \ell_2, \ell_3, \ell_4 \leq s-3, 
  \end{array}
\right.,
\]
where $ \ell_1 + \ell_2 +\ell_3 +\ell_4 =s$.
\end{proof}
Here also, we have obtained an expression for $\reg(J_G^s)$ not in the
form of a linear polynomial. But, as we have demonstrated in the
previous cases, this can always be derived for a given graph.
Analyzing the interplay between the cardinalities of the partitions, one
can obtain the polynomial expression. Let $m_1
= m_2 = m_3 = m_4 = m$. Then
\begin{eqnarray*}
\alpha & = & (s-\ell_4)m_1 +(s-\ell_3)m_2+ (s-\ell_2)m_3 +
(s-\ell_1)m_4 \\
& = & (4s - (\ell_1+\ell_2+\ell_3+\ell_4))m = 3ms.
\end{eqnarray*}
Therefore $\reg(J_G^s) = 3ms + (3m-3)$ for all $s \geq 3$.

\begin{corollary} Let $m,r$ be any two positive integers. Consider the
  arithmetic progression $m_1 = m$, $m_2 =m+r$, $ m_3 =m+2r$, and $m_4
  = m+3r$ in Theorem \ref{4-partite}. Then, for all $s \geq 3$, we have
\[
 \reg(J_G^s) = s (3m+6r)+3m-3.
\]
\end{corollary}
\begin{proof} We have from Theorem \ref{4-partite}, $ \alpha = (s-\ell_4)m_1 +(s-\ell_3)m_2+ (s-\ell_2)m_3 + (s-\ell_1)m_4$. On substituting the values of $m_i$'s in $ \alpha$, we get 
\[
  \alpha = s(4m+6r) -m\ell_4 -(m+r)\ell_3 -(m+2r)\ell_2 - (m+3r)\ell_1
\]
By Theorem \ref{4-partite}, we have for all $s \geq 3$,
\[
\reg(J_G^s) = \max \left\{
  \begin{array}{ll} 
	\alpha,               & \text{ for }  0 \leq \ell_1, \ell_2, \ell_3, \ell_4 \leq s, \\
	\alpha + (m+3r)-1,  & \text{ for }  1 \leq \ell_1 \leq s, \; 0 \leq \ell_2, \ell_3, \ell_4 \leq s-1, \\
	\alpha + 2(m+3r)-2, & \text{ for }  2 \leq \ell_1 \leq s, \; 0 \leq \ell_2, \ell_3, \ell_4 \leq s-2, \\
	\alpha + 3(m+3r)-3, & \text{ for }  3 \leq \ell_1 \leq s, \; 0 \leq \ell_2, \ell_3, \ell_4 \leq s-3, 
  \end{array}
\right.
\]
where $ \ell_1 + \ell_2 +\ell_3 +\ell_4 =s$. To achieve the maximum value of $\alpha$, negative terms in $\alpha$ should 
be minimum. The coefficient of $\ell_1$ in negative terms in $\alpha$ is largest, so $\ell_1$ should be assigned the minimum value. After assigning the minimum value to $\ell_1$, assign the minimum value to $\ell_2$, and similarly minimum value to $\ell_3$. Then assign $ \ell_4 = s-\ell_2 -\ell_3 -\ell_4$. For instance, to get the maximum of $\alpha$ when $1 \leq \ell_1 \leq s, \; 0 \leq \ell_2, \ell_3, \ell_4 \leq s-1$, put $\ell_1=1,\ell_2=0,\ell_3=0$, and $\ell_4=s-1$. With appropriate substitution, we get for all $s \geq 3$
\[
\reg(J_G^s) = \max \left\{
  \begin{array}{ll} 
	s (3m+6r),               & \text{ for }  0 \leq \ell_1, \ell_2, \ell_3, \ell_4 \leq s, \\
	s (3m+6r)+m-1,  & \text{ for }  1 \leq \ell_1 \leq s, \; 0 \leq \ell_2, \ell_3, \ell_4 \leq s-1, \\
	s (3m+6r)+2m-2, & \text{ for }  2 \leq \ell_1 \leq s, \; 0 \leq \ell_2, \ell_3, \ell_4 \leq s-2, \\
	s (3m+6r)+3m-3, & \text{ for }  3 \leq \ell_1 \leq s, \; 0 \leq \ell_2, \ell_3, \ell_4 \leq s-3.
  \end{array}
\right.
\]
Clearly for all $m \geq 1$, and for all $s \geq 3$, we get 
\[
 \reg(J_G^s) = s (3m+6r)+3m-3.
\]
\end{proof}

\subsection{Complete $m$-partite graphs}
Let $G$ be a complete graph on $m$-vertices. Then the cover
ideal $J_G$ of $G$ is generated by $\{x_1\cdots \hat{x}_i
\cdots x_m ~ : ~ 1\leq i \leq m\}$. It follows from Remark 
\ref{rem:depth powers} that $\depth R/J_G^s = 0$ for 
all $s \geq m-1$. Moreover, by Remark \ref{rem:lin res powers}, we know 
that $R/J_G^s$ has linear resolution for all $s \geq 1$. Therefore, the 
minimal graded free resolution of $R/J_G^s$ for all $s \geq m-1$ is of the form 
\[
 0 \to R(-s(m-1)-m+1)^{\beta_m} {\longrightarrow} \to \cdots \to
  R(-s(m-1)-1)^{\beta_2}
 {\longrightarrow} R(-s(m-1))^{\beta_1}
  {\longrightarrow} R \to 0.
\]
Let $g_1,g_2,\ldots,g_m$ be the minimal generators $J_G$. Then the elements in 
$J_G^s$ consists of elements $T_{\ell_1,\ell_2,\dots,\ell_m}=g_1^{\ell_1}g_{2}^{\ell_2}\dots g_{m}^{\ell_m}$ such that 
$ \ell_1 + \ell_2 + \cdots + \ell_m = s$ and $0 \leq \ell_i \leq s$. 
Therefore the total number of elements in $J_G^s$ is same as 
the total number of non-negative integral solution to the linear equation $ \ell_1 + \ell_2 + 
\cdots + \ell_m = s$ which is ${s+m-1 \choose m-1  }$. Hence $\mu (J_G^s)= { s+m-1 \choose m-1}$. 
Therefore $\beta_1 = {s+m-1 \choose m-1}$.

Let $\{e_{\ell_1,\ell_2,\dots,\ell_m} \mid 0 \leq \ell_i
\leq s; \text{ and } \ell_1 + \ell_2 + 
\cdots + \ell_m = s\}$ denote the standard basis for
$R^{\beta_1}$. Let $\partial_1 : R^{\beta_1} \longrightarrow R$ be the
map $\partial_1(e_{\ell_1,\ell_2,\dots,\ell_m}) = T_{\ell_1,\ell_2,\dots,\ell_m}$.
As done in the proofs of Theorems \ref{C3} and \ref{C4}, we can
see that the first syzygy is given by the relations of the form 
\[
 x_i \cdot T_{\ell_1,\ell_2, \cdots, \ell_{i-1}, {\ell_{i}-1
 },{\ell_{i+1}+1},\cdots,\ell_m} -  x_{i+1}\cdot T_{\ell_1,\ell_2,
   \cdots,\ell_{i-1},\ell_{i},\ell_{i+1}, \cdots, \ell_{m}}=0
\]
for each $1 \leq i \leq m-1$. Set $\ell_i - 1 = \ell_i'$ and
$\ell_{i+1}+1 = \ell_{i+1}'$. Then it can be seen that, for each $1
\leq i \leq m-1$, there exist as many such relations as the number of
non-negative integer solutions of $\ell_1 + \cdots + \ell_i' + \cdots
\ell_m = s-1$. Therefore, the total number of such linear relations is
$(m-1){s+m-2 \choose m-1}$.
Therefore $\beta_2 = {m-1 \choose 1}{s+m-2 \choose m-1}$. 
However it is not very difficult to realize that writing down the
higher syzygy relations are quite challenging. Based on Theorems \ref{C3} and \ref{C4} and some of the
experimental results using the computational commutative algebra
package Macaulay 2 \cite{M2}, we propose the following conjecture:

\begin{conjecture} Let $R = K[x_1, x_2, \dots,x_m]$ and let $J$ be the cover ideal of the 
complete graph $K_m$. The minimal graded free resolution of $R/I^s$ for all $s \geq m-1$ is of the form 
\[
  0 \to R(-s(m-1)-m+1)^{\beta_m} {\longrightarrow} \cdots
   \longrightarrow  
  R(-s(m-1)-1)^{\beta_2}
 {\longrightarrow} R(-s(m-1))^{\beta_1}
  {\longrightarrow} R \to 0,
\]
where 
\[
 \beta_i = {m-1 \choose i-1}{ s+m-i \choose m-1}.
\]
\end{conjecture}
Notice that proving the above conjecture will
give the Betti numbers of powers of cover ideals of complete 
$m$-partite graphs. We conclude our article by proposing an expression
for the regularity of powers of the cover ideals of complete
$m$-partite graphs:
\begin{conjecture} Let $G$ denote a complete $m$-partite graph with $V(G) = \sqcup_{i=1}^m
V_i$ and $E(G) = \{\{a,b\} ~ : ~ a \in V_i, b \in V_j, i \neq j \}$. 
Set $V_i = \{x_{i1}, \ldots, x_{in_i}\}$ for $i = 1, \ldots, m$. Let
$J_G \subset R = K[x_{ij} ~ : ~ 1 \leq i \leq m; ~1\leq j \leq n_i]$ denote the cover 
ideal of $G$. Let $ 0 \leq \ell_1,\ell_2,\ldots,\ell_m \leq s$ be integers such that 
$ \ell_1 + \ell_2 + \cdots +\ell_m =s$. 
Set 
\[
 \alpha = s \cdot ( \sum_{i=1}^{m}n_i) -  \sum_{i=1}^{m}n_i \ell_{m+1-i}.
\]
Then for all $s \geq m-1$, one has
\[
\reg(J_G^s) = \max \left\{
  \begin{array}{ll} 
	 \alpha,               & \text{ for }  0 \leq \ell_1, \ell_2, \cdots, \ell_m \leq s, \\
	\alpha + n_m-1,  & \text{ for }  1 \leq \ell_1 \leq s, \; 0 \leq \ell_2, \cdots, \ell_m \leq s-1, \\
	\alpha + 2(n_m-1), & \text{ for }  2 \leq \ell_1 \leq s, \; 0 \leq \ell_2,\cdots, \ell_m \leq s-2, \\
	                   & \vdots\\
	
	\alpha + (m-1)(n_m-1), & \text{ for }  m-1 \leq \ell_1 \leq s, \;
	0 \leq \ell_2, \cdots, \ell_m \leq s-(m-1).
  \end{array}
\right.
\]
\end{conjecture}

 %\nocite*{}
\bibliographystyle{plain}  %% or 
\bibliography{refs_reg}

\end{document}